\documentclass[11pt,reqno,a4paper,openany]{amsart}

\usepackage{amsmath,amsfonts,ams}
\usepackage{latexsym}
\usepackage{amssymb}
\usepackage{color}
\usepackage{epsfig}
\usepackage{graphicx}
\usepackage{subfigure}
\usepackage{mathrsfs}
\usepackage{amscd}
\usepackage{amsthm}
\usepackage{color}
\usepackage{cite}
\usepackage{epstopdf}

\usepackage{}
\usepackage{amssymb}
\usepackage{graphicx}
\usepackage{color}
\usepackage{booktabs}
\usepackage{bigstrut}

\numberwithin{equation}{section}
\newtheorem{proposition}{Proposition}[section]
\newtheorem{definition}{Definition}[section]
\newtheorem{lemma}{Lemma}[section]
\newtheorem{theorem}{Theorem}[section]
\newtheorem{corollary}{Corollary}[section]
\newtheorem{remark}{Remark}[section]

\let\pa=\partial
\let\al=\alpha

\def\C{\mathop{\mathbb C\kern 0pt}\nolimits}
\def\DD{\mathop{\bf D\kern 0pt}\nolimits}
\def\K{\mathop{\bf K\kern 0pt}\nolimits}
\def\N{\mathop{\bf N\kern 0pt}\nolimits}
\def\Q{\mathop{\bf Q\kern 0pt}\nolimits}

\newcommand{\la}{\lambda}

\newcommand{\beq}{\begin{equation}}
\newcommand{\eeq}{\end{equation}}
\newcommand{\ben}{\begin{eqnarray}}
\newcommand{\een}{\end{eqnarray}}
\newcommand{\beno}{\begin{eqnarray*}}
\newcommand{\eeno}{\end{eqnarray*}}

\def\pasdegrille{\let\grille = \pasgrille}

\def\aat#1#2#3{
\divide \dimen1 by 48 \dimen3=\dimen1 \multiply \dimen1 by #1
\advance \dimen1 by -\dimen3 \divide \dimen1 by 101 \multiply
\dimen1 by 100 \divide \dimen2 by \count11 \multiply \dimen2 by #2
\setbox0=\hbox{#3}\ht0=0pt\dp0=0pt
  \rlap{\kern\dimen1 \vbox to0pt{\kern-\dimen2\box0\vss}}\dimen1= \wd1
\dimen2=\ht1}
\def\pasgrille{
\count12= \dimen1 \divide \count12 by 50 \divide \dimen2 by \count12
\count11 =\dimen2 \ \divide \dimen1 by 48
\setlength{\unitlength}{\dimen1} \smash{\rlap{\ }} \dimen1= \wd1
\dimen2=\ht1 }
\def\grille{
\count12= \dimen1 \divide \count12 by 50 \divide \dimen2 by \count12
\count11 =\dimen2 \ \divide \dimen1 by 48
\setlength{\unitlength}{\dimen1}
\smash{\rlap{\graphpaper[1](0,0)(50, \count11)}} \dimen1= \wd1
\dimen2=\ht1 }

\pasdegrille

\makeatletter
\newcommand{\Extend}[5]{\ext@arrow0099{\arrowfill@#1#2#3}{#4}{#5}}
\makeatother

\begin{document}

\title[Scattering for $\dot H^\frac12$-critical wave-Hartree equation]
{Scattering theory for the radial $\dot H^\frac12$-critical wave Equation with a cubic convolution}
\author{Changxing Miao}
\address{Institute of Applied Physics and Computational Mathematics, P. O. Box 8009,\ Beijing,\ China,\ 100088;}
\email{miao\_changxing@iapcm.ac.cn}

\author{Junyong Zhang}
\address{Department of Mathematics and Beijing Key Laboratory on MCAACI, Beijing Institute of Technology, Beijing 100081}
\email{zhang\_junyong@bit.edu.cn}

\author{Jiqiang Zheng}
\address{Universit\'e Nice Sophia-Antipolis, 06108 Nice Cedex 02, France}
\email{zhengjiqiang@gmail.com,  zheng@unice.fr}

\maketitle
\begin{abstract}
In this paper, we study the global well-posedness and scattering for
the wave  equation with a cubic convolution $\partial_{t}^2u-\Delta
u=\pm(|x|^{-3}\ast|u|^2)u$ in dimensions $d\geq4$. We prove that if
the radial solution $u$ with  life-span  $I$ obeys $(u,u_t)\in
L_t^\infty(I;\dot H^{1/2}_x(\R^d)\times\dot H^{-1/2}_x(\R^d))$, then
$u$ is global and scatters.  By the strategy derived from
concentration compactness, we show that the proof of the
global well-posedness and scattering is reduced to disprove the
existence of two scenarios:  soliton-like solution and high to low
frequency cascade.  Making use of the No-waste Duhamel formula and
double Duhamel trick, we
 deduce that these two scenarios enjoy the additional
 regularity by the bootstrap argument of \cite{DL}. This  together with virial analysis  implies
the energy of such  two scenarios is zero and so we get a
contradiction.
\end{abstract}

\begin{center}
 \begin{minipage}{120mm}
 \end{minipage}
 \end{center}

\section{Introduction}
\setcounter{section}{1}\setcounter{equation}{0} This paper is
devoted to the study of the Cauchy problem of the energy-subcritical
wave equation with a cubic convolution
\begin{align} \label{equ1.1}
\begin{cases}    \partial_{t}^2u-\Delta u+f(u)=0,  \qquad  (t,x) \in
\mathbb{R}\times\mathbb{R}^d,~~d\geq4\\
(u,u_t)(t_0,x)=(u_0,u_1)(x)\in \dot H^{1/2}(\R^d)\times\dot
H^{-1/2}(\R^d),\end{cases}
\end{align}
where $f(u)=\mu(V(\cdot)\ast|u|^2)u$ with $V(x)=|x|^{-3}$, $\mu=\pm
1$ with $\mu=1$ known as the defocusing case and $\mu=-1$ as the
focusing case. Here $u$ is a real-valued function defined in
$\mathbb{R}^{d+1}$, $\Delta$ is the Laplacian in $\mathbb{R}^{d}$,
$V(x)$ is called the potential, and $*$ denotes the spatial
convolution in $\mathbb{R}^{d}$. Especially for $d=5$,
$V(x)=|x|^{-3}$ is the Newtonian potential.

If the solution $u$ of \eqref{equ1.1} has sufficient decay at
infinity and smoothness, it conserves energy:
\begin{equation*}
E(u,u_t)=\frac12\int_{\mathbb{R}^d}\big(|\pa_tu|^2+|\nabla
u|^2\big)\mathrm{d}x+\frac\mu4\iint_{\mathbb{R}^d\times\mathbb{R}^d}\frac{|u(t,x)|^2|u(t,y)|^2}{|x-y|^{3}}\mathrm{d}x\mathrm{d}y=E(u_0,u_1).
\end{equation*}

The  equation \eqref{equ1.1} is  invariant under the scaling
\begin{equation}\label{scale}
(u,u_t)(t,x)\mapsto \big(\la^{\frac{d-1}2}\ u(\la t,\la
x),\la^{\frac{d+1}2}\ u_t(\la t,\la x)\big), ~~\forall~\la>0.
\end{equation}
One can verify that the only homogeneous $L^2_x$-based Sobolev space
that is left invariant under \eqref{scale} is  $\dot
H^\frac12(\R^d)\times\dot H^{-\frac12}(\R^d)$. And so, it is natural
to consider the Cauchy problem with initial data $(u_0,u_1)(x)\in
 \dot{H}^\frac12(\R^d)\times\dot H^{-\frac12}(\R^d)$. We will prove that any
radial maximal-lifespan solution that remains uniformly bounded in
 $\dot H^\frac12(\R^d)\times\dot H^{-\frac12}(\R^d)$ must be global and
scatter.

First, we introduce some definitions and background materials.
\begin{definition}[Strong solution]\label{def1.1}  Let $I$ be a nonempty time interval including $t_0$.
 A function $u:~I\times\R^d\to\R$ is a strong solution to problem \eqref{equ1.1},
if
$$(u,u_t)\in C_t^0(J;\dot H^\frac12(\R^d)\times\dot H^{-\frac12}(\R^d), \;\text{and}\; u\in
L_{t,x}^\frac{2(d+1)}{d-1}(J\times\R^d)$$
for any compact $J\subset
I$ and for each $t\in I$ such that
\begin{equation}\label{duhamel}
{u(t)\choose \dot{u}(t)} = V_0(t-t_0){u_0(x) \choose u_1(x)}
-\int^{t}_{t_0}V_0(t-s){0 \choose f(u(s))} \mathrm{d}s,
\end{equation}
where
\begin{equation}\label{v0tdefin}
V_0(t) = {\dot{K}(t)\quad K(t) \choose \ddot{K}(t)\quad \dot{K}(t)},
\quad K(t)=\frac{\sin(t\omega)}{\omega},\quad
\omega=\big(-\Delta\big)^{1/2},
\end{equation} and the dot denotes the time derivative.
 We refer to the interval $I$ as
the lifespan of $u$. We say that $u$ is a maximal-lifespan solution
if the solution cannot be extended to any strictly large interval.
We say that $u$ is a global solution if $I=\R.$
\end{definition}

The solution lies in the space $
L_{t,x}^\frac{2(d+1)}{d-1}(I\times\R^d)$ locally in time is natural
since by Strichartz estimate in  Lemma \ref{strichartzest} below,
the linear flow always lies in this space.  We define the scattering
size of a solution $u$ to problem \eqref{equ1.1} on a time interval
$I$ by
\begin{equation}\label{scattersize1.1}
 \|u\|_{S(I)}:=\|u\|_{
L_{t,x}^\frac{2(d+1)}{d-1}(I\times\R^d)}=:S_I(u).
\end{equation}
Standard arguments show that if a solution $u$ of problem \eqref{equ1.1} is
global, with $S_{\R}(u)<+\infty$, then it scatters, i.e.  there
exist unique
$(v_0^\pm,v_1^\pm)\in\dot H^\frac12(\R^d)\times\dot H^{-\frac12}(\R^d)$ such
that
\begin{equation}\label{1.2.1}
\bigg\|{u(t)\choose \dot{u}(t)}-V_0(t){v_0^\pm\choose
v_1^\pm}\bigg\|_{\dot H^\frac12(\R^d)\times\dot H^{-\frac12}(\R^d)}
\longrightarrow 0,\quad \text{as}\quad t\longrightarrow \pm\infty.
\end{equation}

\vskip0.1cm

 The notion closely associated with
scattering is the definition of blowup:
\begin{definition}[Blowup]\label{def1.2.1} We call that a  maximal-lifespan solution
$u:I\times\R^d\to\mathbb{R}$ of problem \eqref{equ1.1} blows up forward in time if there
exists a time $\tilde{t}_0\in I$ such that $\|u\|_{S([\tilde{t}_0,\sup
I))}=+\infty$. Similarly, $u(t,x)$ blows up backward in time if there
exists a time $\tilde t_0\in I$ such that
$\|u\|_{S(\inf I,\tilde{t}_0]}=+\infty$.
\end{definition}

Now we state our main result.

\begin{theorem}\label{theorem}
Assume that $d\geq4$. Let $u:~I\times\R^d\to\mathbb{R}$ be a radial
maximal-lifespan solution to problem \eqref{equ1.1} such that
\begin{equation}\label{assume1.1}
\big\|(u,u_t)\big\|_{L_t^\infty(I;
\dot H^\frac12(\R^d)\times\dot H^{-\frac12}(\R^d))}<+\infty.
\end{equation} Then the solution $u$ is global and scatters in the sense of \eqref{1.2.1}.
\end{theorem}
\begin{remark} We here focus on the $\dot H^{\frac12}$-critical problem, which is corresponding to $V(x)=|x|^{-\gamma}$ with $\gamma=3$.
It is a natural assumption on the dimension $d>\gamma=3$ for the
convolution in the nonlinearity. This is why we are restricted to
$d\geq 4$. The radial assumption comes from the technique issue
improving the regularity in Section 4 below.
\end{remark}

The impetus to consider this problem stems from a series of recent
works for the energy-critical, energy-supercritical and
energy-subcritical nonlinear wave equation. To be more precise, let us recall the results for the
 nonlinear wave equation(NLW)
 \begin{equation}\label{waveequ}
u_{tt}-\Delta u+\mu|u|^pu=0,~(t,x)\in\R\times\R^d.
 \end{equation}
 The equation \eqref{waveequ} is
called energy-critical if $p=\tfrac4{d-2}$, energy-supercritical if
$p>\tfrac4{d-2}$ and energy-subcritical if $p<\tfrac4{d-2}$. The
energy-critical equations have been the most widely studied
instances of NLW, since the rescaling
\begin{equation}
(u,u_t)(t,x)\mapsto~\big(\la^\frac2pu(\la t,\la
x),\la^{\frac2p+1}u_t(\la t,\la x)\big), \;\;p=\frac4{d-2}
\end{equation}
leaves invariant the energy of solutions, which is defined by
$$E(u,u_t)=\frac12\int_{\R^d}\big(|\pa_tu|^2+|\nabla
u|^2\big)\mathrm{d}
x+\frac{\mu}{p+2}\int_{\R^d}|u|^{p+2}\mathrm{d}x,$$ and is a
conserved quantity for equation\eqref{waveequ}.  For the defocusing
energy-critical NLW \eqref{waveequ}, Grillakis \cite{Gri90} proved that the Cauchy problem of equation \eqref{waveequ} with the  $\dot
H^1(\mathbb R^3)\times L^2(\mathbb R^3)$ intial data is  global well posedness, we  refer the readers to \cite{BG},  \cite{Kapi94},
\cite{ShaStr94} and  \cite{T07} for the scattering theory and the high dimensional case. In particular, Tao
derived a exponential type spacetime bound in \cite{T07}. In the
above papers, their methods rely heavily on the classical finite speed of
propagation (i.e. the monotonic local energy estimate on the light
cone)
\begin{align}\label{}
\int_{|x|\leq R-t}e(t,x)dx\leq \int_{|x|\leq R }e(0,x)dx,~~t>0
\end{align}where\begin{align}\label{}
  e(t,x):=\frac12  | u_t
 |^2  + \frac12  | \nabla u  |^2   + \frac{d-2}{2d} |
u|^\frac{2d}{d-2},
\end{align}
and Morawetz estimate
\begin{eqnarray}\label{morawetzes}
\int_{\R}\int_{\R^d}\frac{|u|^{\frac{2d}{d-2}}}{|x|} dxdt\leq
C\big(E(u_0,u_1)\big).
\end{eqnarray}
However, the Morawetz estimate fails for the focusing energy-critical NLW.
 Kenig and Merle \cite{KM1} first employed
sophisticated ``concentrated compactness + rigidity method" to
obtain the dichotomy-type result under the assumption that $E (u_0,
u_1) < E (W, 0)$, where $W$ denotes the ground state of the elliptic
equation
$$\Delta W+|W|^\frac4{d-2}W=0.$$ The
analogs for the  focusing energy-critical nonlinear Schr\"{o}dinger
equation in the radial case for dimensions $3$ and $4$ have also been
established by Kenig and Merle \cite{KM}. Thereafter,  Bulut et.al \cite{BCLPZ}
extended the above result in \cite{KM1} to higher dimensions. While we
refer the readers to Killip and Visan \cite{KV} for the focusing energy-critical
Schr\"odinger equation in high dimensions. This was proven by making use of minimal counterexamples derived from the
concentration-compactness approach to induction on energy.

For the $\dot H^{s_c}$-critical NLW \eqref{waveequ} with
$p\neq\tfrac4{d-2}$, both  the Morawetz estimate and energy
conservation fails. It is hard to prove  global well posedness and
scattering of equation \eqref{waveequ}
 in  the space $\dot H^{s_c}\times \dot{H}^{s_c-1}$. Up to now, we do not know  how to
treat the large-data case  since there does not exist  any a priori control of a
critical norm. The first result in this
direction is due to Kenig-Merle \cite{KM2010}, where they studied
the $\dot H^\frac12$-critical Schr\"odinger equation in $\R^3$. For
the defocusing energy-supercritical NLW in odd dimensions, Kenig and
Merle \cite{KM2011,KM2011D} proved that if the radial solution $u$
is apriorily bounded in the critical Sobolev space, that is
$$(u,u_t)\in L_t^\infty(I; \dot{H}^{s_c}_x(\R^d)\times
\dot{H}^{s_c-1}_x(\R^d)), \quad s_c:=\tfrac{d}2-\tfrac2p>1,$$ then $u$
is global and scatters.  Later, Killip and Visan \cite{KV} showed
the result in $\R^3$ for the non-radial solutions by making use of
Huygens principal and so called ``localized double Duhamel trick".
We refer to \cite{Blut2012,KV2011,MWZ} for some high dimensional
cases.
 Recently, Duyckaerts, Kenig and Merle \cite{DKM} obtain such result
 for the focusing energy-supercritical NLW with radial solution
in three dimension. Their proof relies on the compactness/rigidity
method, pointwise estimates on compact solutions obtained in
\cite{KM2011}, and the channels of energy method  developed in \cite{DKM2011}.
Furthermore, by exploiting the double Duhamel trick, Dodson and
Lawrie\cite{DL14} extend the result in \cite{DKM} to dimension five.
We also refer reader to \cite{KV2010,MMZ,Mur} for the defocusing
energy-supercritical Schr\"odinger equation.

 The methods employed in the energy-supercritical NLW
 also lead to the study of the energy-subcritical
NLW. In fact, using  the channels of energy method pioneered in
\cite{DKM2012,DKM}, Shen \cite{Shen} proved the analog result of
\cite{KM2011} for $2<p<4$ with $d=3$ in both defocusing and focusing
case. However, the channels of energy  method does not work so
effectively for $p\leq2$. Recently, by virial based rigidity
argument and improving addition regularity for the minimal
counterexamples,  Dodson and Lawrie \cite{DL} extended the result of
\cite{Shen} to $\sqrt{2}<p\leq2$. The problem of having monotonicity
formulae at a different regularity to the critical conservation laws
is a difficulty intrinsic to the nonlinear wave equation. In order
to enable to utilize the monotonicity formulae, one need improve the
regularity for the almost periodic solutions. In \cite{DL,Shen},
they use the double Duhamel trick to show that almost periodic
solutions belong to energy space $\dot H^1_x(\R^3)\times
L^2_x(\R^3)$. The main difficult is that the decay rate of the
linear solution is not enough to guarantee the double Duhamel
formulae converges. However, the weighted decay available from
radial Sobolev embedding can supply the additional decay to
guarantee the double Duhamel formulae converges. Thus, one need the
radial assumption in \cite{DL,Shen}.
 This is different
 from cubic Schr\"odinger equation \cite{KM2010}, where no radial assumption is made. This is due to the Lin-Strauss Morawetz inequality
$$\int_{I}\int_{\R^{3}}\frac{|u(t,x)|^4}{|x|}\;dx\;dt\lesssim\sup_{t\in I}\|u(t,x)\|_{\dot H^{\frac12}_x(\R^3)}^2,$$
which is scaling critical with the cubic Schr\"odinger equation.

For the Cauchy problem of the nonlocal NLW \eqref{equ1.1} with
$V(x)=|x|^{-\gamma}$, making use of  the argument developed by  Strauss \cite{St81a,St81b}
and Pecher \cite{Pe85}, Mochizuki \cite{Mo89} showed that if $d \geq
3$, $2\leq \gamma < \min\{d, 4\}$, then the global well-posedness
and scattering results with small initial data hold in the energy
space $H^1(\R^d)\times L^2(\R^d)$. For the large initial data, it is difficult since
 there are  the absence of the
classical finite speed of propagation for  \eqref{equ1.1}  and the positive properties $G(u)>0$ with
$$G(u)=f(u){\bar u}-2\int_0^{|u|}f(r)dr, \; f(u)=(V*|u|^2)u, \; V(x)=|x|^{-\gamma}.$$
It plays an important role in establishing the classical
 Morawetz-type estimates
 in \cite{Na99b}.
Hence, one can not utilize
the classical methods in \cite{Gri90, Kapi94,ShaStr94} to prove the
GWP  scattering for the defocusing energy-critical wave equation \eqref{equ1.1} with $\gamma=4$.
 While, inspired by the strategy derived from
concentration compactness \cite{KM,KM1} and the new-type  Morawetz-type estimate in
 \cite{Na99e}, the authors in
\cite{MZZ} showed GWP and scattering result simultaneously for the
defocusing energy-critical wave-Hartree equation \eqref{equ1.1} with
$V(x)=|x|^{-4}$.

 In this paper, we continue the investigations carried out in \cite{MZZ} concerning the long-time behavior of the
 solution of wave-Hartree, but for $\dot H^{1/2}$-critical wave-Hartree (that is, $\gamma=3$) in both defocusing and focusing case.
  Compared with \cite{MZZ},
the $\dot H^{1/2}$-critical problem is much more difficult due to
the failure of energy conservation law and the Morawetz estimate. By
using compactness approach, modifying the argument of improving
addition regularity as in \cite{DL} and employing the symmetries of
the non-local nonlinearity $(|x|^{-3}\ast|u|^2)u$, we prove that if
the radial solution $u$ with life-span  $I$ satisfies $(u,u_t)\in
L_t^\infty(I;\dot H^\frac12(\R^d)\times\dot H^{-\frac12}(\R^d))$, then
$u$ is global and scatters. \vspace{0.2cm}

Now, let us turn to an outline of the arguments establishing Theorem
\ref{theorem}.

\vskip 0.2in

\noindent{\bf $\bullet$ The outline of the proof of Theorem \ref{theorem}.}
Before we can address the global-in-time theory for the problem
\eqref{equ1.1}, we need to have a good local-in-time theory in
place. In particular, we have the following

\begin{theorem}[Local well-posedness]\label{local}
Let $d\geq4$. Assume that  $(u_0,u_1)\in \dot H^\frac12(\R^d)\times\dot H^{-\frac12}(\R^d)$ and $t_0\in\R$, then there
exists a unique maximal-lifespan solution $u: I\times\R^d\to\R$ to problem \eqref{equ1.1} with initial data
$\big(u_0,u_1\big)$. This solution also
has the following properties:
\begin{enumerate}
\item {\rm(Local existence)}\; $I$ is an open neighborhood of $t_0$.

\item {\rm(Blow up criterion)}\; If\ $\sup (I)$ is finite, then $u$ blows up
forward in time in the sense of Definition \ref{def1.2.1}. If $\inf
(I)$ is finite, then $u$ blows up backward in time.

\item {\rm(Scattering)}\; If $\sup (I)=+\infty$ and $u$ does not blow up
forward in time, then $u$ scatters forward in time in the sense
\eqref{1.2.1}. Conversely, given $(v_0^+,v_1^+)\in \dot H^\frac12(\R^d)\times\dot H^{-\frac12}(\R^d)$ there is a unique solution
to problem \eqref{equ1.1} in a neighborhood of infinity such that
\eqref{1.2.1} holds.

\item There
exists a $\delta=\delta(d,\|(u_0,u_1)\|_{\dot{H}^{\frac12}\times
\dot{H}^{-\frac12}})$ such that if
$$\big\|\dot{K}(t-t_0)u_0 +
K(t-t_0)u_1\big\|_{S(I)}<\delta,$$ then, there exists a unique
solution $u:~I\times\mathbb{R}^d\to\R$ to the equation in \eqref{equ1.1} with initial data $(u(t_0),\dot{u}(t_0))$ such that
$(u,\dot{u})\in C(I;\dot H^\frac12(\R^d)\times\dot H^{-\frac12}(\R^d)),$ and
\begin{equation*}
\|u\|_{S(I)}\leq 2\delta,~ \|u\|_{S^\frac12(I)}<+\infty,
\end{equation*}
where $\|u\|_{S(I)}$ is defined in \eqref{scattersize1.1} and
$$\|u\|_{S^\frac12(I)}:=\sup_{(q,r)\in\Lambda_0}\big\{\big\||\nabla|^su\big\|_{L_t^q(I;L_x^r)}:~
\tfrac1q+\tfrac{d}r=\tfrac{d}2-\big(\tfrac12-s\big),~0\leq
s\leq\tfrac12\big\},$$ with $\Lambda_0$ being defined in Definition
2.1 below.

\item {\rm(Small data implying scattering)}\; If
$\big\|(u_0,u_1)\big\|_{\dot{H}^{\frac12}\times \dot{H}^{-\frac12}}$ is
sufficiently small, then $u$ is a global solution which does not
blow up either forward or backward in time,  and satisfies that
$$\|u\|_{S(\R)}\lesssim\big\|(u_0,u_1)\big\|_{\dot{H}^{\frac12}\times \dot{H}^{-\frac12}}.$$
\end{enumerate}
\end{theorem}

Theorem \ref{local} follows from Strichartz estimate in Lemma
\ref{strichartzest} below and the standard fixed argument in
\cite{cwI}. Closely related to the continuous dependence on the
data, an essential tool for concentration compactness
 arguments is the following stability theory.
\begin{lemma}[Stability]\label{long}  Let $I$ be a
time interval, and let $\tilde{u}$ be a function on $I\times \R^d$
which is  a near-solution to problem \eqref{equ1.1}  in the sense that
\begin{equation}\label{near}
\tilde{u}_{tt}-\Delta \tilde u=-f(\tilde{u})+e
\end{equation}
for some function $e$. Assume that
\begin{equation}\label{eq2.20}\begin{aligned}
\|\tilde{u}\|_{S(I)}\leq M,\\
\big\|(\tilde{u},\partial_{t}\tilde{u})\big\|_{L^\infty_t (I;\dot
H^{\frac12}\times\dot H^{-\frac12})}\leq E
\end{aligned}
\end{equation}
for some constant $M,E>0$, where $S(I)$ is defined in
\eqref{scattersize1.1}. Let $t_0\in I$, and let
$(u(t_0),u_t(t_0))\in \dot H^\frac12(\R^d)\times\dot H^{-\frac12}(\R^d)$ be close to
$(\tilde{u}(t_0),\tilde{u_t}(t_0))$ in the sense that
\begin{equation}\label{eq2.21}
\big\|\big(u(t_0)-\tilde{u}(t_0),u_t(t_0)-\tilde{u}_t(t_0)\big)\big\|_{\dot
H^{\frac12}\times \dot H^{-\frac12}}\leq\epsilon
\end{equation}
and assume also that the error term obeys
\begin{equation} \label{eq2.22}
\begin{split}
\|e\|_{L_{t,x}^\frac{2(d+1)}{d+3}(I\times\R^d)}\leq \epsilon
\end{split}
\end{equation}
for some small $0<\epsilon<\epsilon_1=\epsilon_1( M, E)$. Then, we
conclude that there exists a solution $u:~I\times\R^d\to\R$ to problem
\eqref{equ1.1} with initial data $(u(t_0),u_t(t_0))$ such that
\begin{equation}\label{eq2.23}
\begin{aligned}
\|u-\tilde{u}\|_{L_t^q(I,L_x^r)}\leq & C(M,E),\\
\|u\|_{L_t^q(I,L_x^r)}\leq & C(M,E),\\
\|u-\tilde{u}\|_{S(I)}\leq & C(M,E)\epsilon^c,
\end{aligned}
\end{equation}
where constant $c=c(d, M, E)>0$  and
$(q,r)$ is admissible pair and satisfies
$\frac1q+\frac{d}r=\frac{d}2-\frac12$.
\end{lemma}

Lemma \ref{long} follows from well-known arguments, which are in
fact similar in spirit to the arguments used to prove local
well-posedness. For an example of such an argument in the context of
wave equation with a cubic convolution, see for example \cite{MZZ}.

\vskip 0.2in

With the local theory in hand, we are now in a position to sketch
the proof of Theorem \ref{theorem} by the compactness procedure
as in \cite{KM1, KVnote}. For any $0\leq E_0<+\infty,$ we define
$$L(E_0):=\sup\Big\{S_I(u):~u:~I\times\R^d\to \R\ \text{such\ that}\
\sup_{t\in I}\big\|(u,u_t)\big\|_{\dot
H^{\frac12}\times\dot{H}^{-\frac12}}^2\leq E_0\Big\},$$
where the supremum is taken over all solutions $u:~I\times\R^d\to
\mathbb{R}$ to problem \eqref{equ1.1} satisfying $\sup\limits_{t\in
I}\big\|(u,u_t)\big\|_{\dot
H^{\frac12}\times\dot{H}^{-\frac12}}^2\leq E_0.$ Thus, $L:\
[0,+\infty)\to [0,+\infty)$ is a non-decreasing function. Furthermore,
 we obtain from ``Small data implying scattering"
$$L(E_0)\lesssim E_0^\frac12\quad \text{for}\quad E_0\leq\eta_0^2.$$

\vskip0.1cm

From Lemma \ref{long}, we see that $L$ is continuous. Therefore,
there must exist a unique critical element $E_c\in(0,+\infty]$ such that
$L(E_0)<+\infty$ for $E_0<E_c$ and $L(E_0)=+\infty$ for $E_0\geq
E_c$. In particular, if $u:~I\times\R^d\to \R$ is a maximal-lifespan
solution to problem \eqref{equ1.1} such that $\sup\limits_{t\in
I}\big\|(u,u_t)\big\|_{\dot
H^{\frac12}\times\dot{H}^{-\frac12}}< E_c,$ then $u$ is global
and moreover,
$$S_\R(u)\leq L\big(\big\|(u,u_t)\big\|_{L_t^\infty(\R;\dot H^{\frac12}\times\dot{H}^{-\frac12})}\big).$$
Therefore, the proof of Theorem \ref{theorem} is equivalent to show
$E_c=+\infty.$  We argue by contradiction. The failure of Theorem
\ref{theorem} would imply the existence of very special class of
solutions. Our goal is to exclude such special class of solutions.
Before making further reductions, we give the definition of the almost periodic
solution.

\begin{definition}[Almost periodic solution]\label{apms}
Let $d\geq4,$ a solution $(u,u_t)(t)$ to problem \eqref{equ1.1} with
lifespan $I$ is called almost periodic modulo symmetries if
$(u,u_t)(t)$ is bounded in
$\dot{H}_x^{1/2}(\R^d)\times\dot{H}_x^{-1/2}(\R^d)$ and there exist
functions $N(t):~I\to\R^+,~x(t):~I\to\R^d$ and $C(\eta):\R^+\to\R^+$
such that for all $t\in I$ and $\eta>0$,
\begin{equation}\label{apss}
\int_{|x-x(t)|\geq\frac{C(\eta)}{N(t)}}\Big(\big||\nabla|^{1/2}
u(t,x)\big|^2+\big||\nabla|^{-1/2}
u_t(t,x)\big|^2\Big)\mathrm{d}x\leq\eta
\end{equation}
and
\begin{equation}\label{apsf}
\int_{|\xi|\geq C(\eta)N(t)}\Big(|\xi|\cdot\big|\hat
u(t,\xi)\big|^2+|\xi|^{-1}\cdot\big|\hat
u_t(t,\xi)\big|^2\Big)\mathrm{d}\xi\leq\eta.
\end{equation}
We refer to the function $N(t)$ as the frequency scale function of
the solution $u$,  to $x(t)$ as the spatial center function, and to
$C(\eta)$ as the compactness modules function.
\end{definition}

We remark that for radial almost periodic solutions, one can take
the spatial center function $x(t)\equiv0$, see \cite{KVnote}.

By using the  Bahouri-G\'erard type profile decomposition from
\cite{BG,Bu2010} and the above stability result, we deduce that the
failure of Theorem \ref{theorem} would imply the existence of the
radial almost periodic solutions.

\begin{theorem}[Reduction to radial almost periodic solutions,
\cite{KM,KV2010,MXZ09}]\label{reduce} Assume $E_{c}<+\infty.$ Then
there exists a radial maximal-lifespan solution
$u:~I\times\R^d\to\R$ to problem \eqref{equ1.1} such that
\begin{enumerate}
\item $u$ is radial almost periodic modulo symmetries;

\item $u$ blows up both forward and backward in time;

\item $u$ has the minimal $L_t^\infty\dot{H}^{\frac12}$-norm among all blowup solutions. More precisely, let
$v:~J\times\R^d\to\R$ be a maximal-lifespan solution which blows up
in at least one time direction, then
$$\sup_{t\in I}\big\|(u,u_t)(t)\big\|_{\dot{H}^{\frac12}\times\dot{H}^{-\frac12}}\leq\sup_{t\in
J}\big\|(v,v_t)(t)\big\|_{\dot{H}^{\frac12}\times\dot{H}^{-\frac12}}.$$
\end{enumerate}
\end{theorem}
The reduction to almost periodic solutions is now widely regarded as
a standard technique in the study of dispersive equations at
critical regularity.  Their existence was first proved in the
pioneering work by Karaani \cite{Ker} for the mass-critical NLS.
Kenig and Merle \cite{KM} adapted the argument to the
energy-critical NLS, and first applied this to study the
wellposedness  and scattering problem. Since then, the technique has
proven to be extremely useful,  see  \cite{KVnote,KV,KV2011,MXZ09}
for many more examples of these techniques. In particular, for a
very nice introduction to concentration compactness methods, one should refer to
\cite{KVnote,Visan}.

With Theorem \ref{reduce} in place, we can now make some refinements
to the class of solutions that we consider. A rescaling argument and
possibly time reversal as in \cite{KTV,TVZ} show that we can
restrict our attention to radial almost periodic solutions that do
not escape to arbitrarily high frequencies on at least half of their
maximal lifespan $[0,\infty)$.

\begin{theorem}[Two enemies, \cite{KVnote,MXZ09,TVZ}]\label{three} Suppose
Theorem \ref{theorem} fails. Then there exists a radial
maximal-lifespan solution $u:~(T_-,T_+)\times\R^d\to\R$, which is
radial almost periodic modulo symmetries with
$S_{(T_-,T_+)}(u)=+\infty.$ Moreover, we can also ensure that
$T_+(u)=+\infty,~T_-<0$ and the frequency scale function $N(t)$
matches one of the following two scenarios:
\begin{enumerate}
\item$($Soliton-like solution$)$ We have  $N(t)=1$ for all $t\in\R.$

\item$($High-to-low frequency cascade$)$ We have
$$\sup_{t\in\R^+}N(t)\leq1,~\text{and}~\varliminf_{t\to\infty}N(t)=0.$$
\end{enumerate}
\end{theorem}

 In view of this theorem, our
goal is to preclude the possibilities of all the scenarios.

A further manifestation of the minimality of $u$ as a blow-up
solution is the absence of the scattered wave at the endpoints of
the lifespan $I$.  Formally, we have the following Duhamel
formula, which plays an important role in proving the additional
 regularity. One can refer to \cite{CKSTT07,KVnote,MXZ09} for the proof.

\begin{lemma} [No waste Duhamel formula]\label{lemnwdf} Let $u: I\times\R^d\to\R$ be a maximal-lifespan solution to the equation in \eqref{equ1.1} which is
almost periodic modulo symmetries. Then, for all $t\in I,$ there
holds that
\begin{equation}\label{nwd}
\begin{split}
{u(t)\choose u_t(t)}=&\lim_{T\nearrow\sup(I)}\int_t^{T}V_0(t-s){0\choose f(u)(s)}\,\mathrm{d}s\\
=&-\lim_{T\searrow\inf(I)}\int_T^tV_0(t-s){0\choose
f(u)(s)}\,\mathrm{d}s,
\end{split}
\end{equation}
as weak limits in
$\dot{H}_x^{1/2}(\R^d)\times\dot{H}_x^{-1/2}(\R^d)$. Here $V_0(t)$
is defined as in \eqref{v0tdefin}.
\end{lemma}

In view of this lemma and note that the minimal
$L_t^\infty(\dot{H}^{\frac12}_x\times\dot H^{-\frac12}_x)$-norm blowup solution is localized in both physical and frequency space,
we can show that it admits additional regularity by the bootstrap
argument and double Duhamel trick.

\begin{theorem}[Additional Regularity]\label{regular}  Let
$u:~(T_-,+\infty)\times\R^d\to\R$ be a radial  solution to equation
in \eqref{equ1.1} which is almost periodic modulo symmetries in the
sense of Theorem \ref{three}. And assume that $N(t)\leq1$ on
$t\in\R^+$, then for each $t\in\R^+$, there holds
\begin{equation}\label{addrgrlt}
\|(u,u_t)\|_{\dot H^{1}_x(\R^d)\times L^2_x(\R^d)}\lesssim
\begin{cases}
N(t)^\frac12\quad\text{if}\quad d\geq5,\\
N(t)^\frac13\quad\text{if}\quad d=4.
\end{cases}
\end{equation}
\end{theorem}

It follows from this theorem that the high-to-low frequency cascade
solution satisfies
$$\varliminf_{t\to\infty}\|(u,u_t)\|_{\dot H^{1}_x(\R^d)\times L^2_x(\R^d)}\lesssim \varliminf_{t\to+\infty}
N(t)^\frac13=0,$$ which implies that
$$E(u,u_t)(t)\to0,~\text{as}~t\to+\infty.$$
This together with Theorem \ref{propbnpe} below precludes the
existence of the high-to-low frequency cascade solution.
\vskip0.1cm

Moreover, for the soliton-like solution, we have
\begin{proposition}\label{propgdzzx}
 Let
$u:~\R\times\R^d\to\R$  be a radial  solution to equation in \eqref{equ1.1}
which is almost periodic modulo symmetries in the sense of Theorem
\ref{three} with $N(t)\equiv1$. Then the set
$K:=\big\{(u,u_t)(t):~t\in\R\big\}$ is precompact in $(\dot
H^\frac12\times\dot H^{-\frac12})\cap(\dot H^1\times L^2)$.

\end{proposition}

We will utilize this proposition and the virial identity to prove
that the energy of such solution is exact zero. Thus, we exclude the
existence of the soliton-like solution. We refer to Section 3 for more
details.

The paper is organized as follows. In Section $2$,  as
preliminaries, we gather notations and recall the Strichartz
estimate for wave equation and some useful lemmas.  In Section $3$,
we exclude two scenarios in the sense of Theorem \ref{three} under
the assumption that Theorem \ref{regular} and Proposition
\ref{propgdzzx} hold. In Section $4$, we show  Theorem \ref{regular}
and Proposition \ref{propgdzzx}, and so we conclude Theorem
\ref{theorem}.

\section{Preliminaries}
 \setcounter{section}{2}\setcounter{equation}{0}

\subsection{Notations and useful lemmas}
First, we give some notations which will be used throughout this
paper. To simplify  our inequalities, we introduce
the  symbols $\lesssim, \thicksim, \ll$. If $X, Y$ are nonnegative
quantities, we write either $X\lesssim Y $ or $X=O(Y)$ to denote the estimate
$X\leq CY$ for some $C$, and $X \thicksim Y$ to denote the estimate
$X\lesssim Y\lesssim X$. We use $X\ll Y$ to mean $X \leq c Y$ for
some small constant $c$. We use $C\gg1$ to denote various large
finite constants, and $0< c \ll 1$ to denote various small
constants.  For every $r, 1\leq r \leq \infty$, we denote by $\|\cdot
\|_{r}$ the norm in the Lebesque space $L^{r}=L^{r}(\mathbb{R}^d)$ and by $r'$ the
conjugate exponent
 defined by $\frac{1}{r} + \frac{1}{r'}=1$.

The Fourier transform on $\mathbb{R}^d$ is defined by
\begin{equation*}
\aligned \widehat{f}(\xi):= \big( 2\pi
\big)^{-\frac{d}{2}}\int_{\mathbb{R}^d}e^{- ix\cdot
\xi}f(x)\,\mathrm{d}x ,
\endaligned
\end{equation*}
giving rise to the fractional differentiation operator
$|\nabla|^{s}$  defined by
\begin{equation*}
\aligned |\nabla|^sf(x):=\mathcal
F^{-1}_{\xi}(|\xi|^s\hat{f}(\xi))(x).
\endaligned
\end{equation*}
This helps us to define the homogeneous and inhomogeneous Sobolev
norms
\begin{equation*}
\big\|f\big\|_{\dot{W}^{s,p}_x(\R^d)}:= \big\| |\nabla|^s
f\big\|_{L^p_x(\R^d)}.
\end{equation*}

We will also need the Littlewood-Paley projection operators.
Specifically, let $\varphi(\xi)$ be a smooth bump function adapted
to the ball $|\xi|\leq 2$ which equals 1 on the ball $|\xi|\leq 1$.
For $k\in {\mathbb{Z}}$, we define the Littlewood-Paley operators
\begin{equation*}
\aligned \widehat{P_{\leq k}f}(\xi)& :=
\varphi\Big(\frac{\xi}{2^k}\Big)\widehat{f}(\xi), \\
\widehat{P_{> k}f}(\xi)& :=
\Big(1-\varphi\Big(\frac{\xi}{2^k}\Big)\Big)\widehat{f}(\xi), \\
\widehat{P_{k}f}(\xi)& :=
\Big(\varphi\Big(\frac{\xi}{2^k}\Big)-\varphi\Big(\frac{2\xi}{2^k}\Big)\Big)\widehat{f}(\xi).
\endaligned
\end{equation*}
Similarly we can define $P_{<k}$, $P_{\geq k}$, and $P_{m<\cdot\leq
k}=P_{\leq k}-P_{\leq m}$.

The Littlewood-Paley operators commute with derivative operators,
the free propagator, and the conjugation operation. They are
self-adjoint and bounded on every $L^p_x$ and $\dot{H}^s_x$ space
for $1\leq p\leq \infty$ and $s\geq 0$, moreover, they also obey the
following
 Bernstein estimates
\begin{eqnarray*}\label{bernstein}
\big\||\nabla|^{\pm s} P_{k} f \big\|_{L^p} & \thicksim & 2^{\pm ks}
\big\|
P_{k} f \big\|_{L^p},  \\
\big\| P_{\leq k} f \big\|_{L^q} & \lesssim &
2^{(\frac{d}{p}-\frac{d}{q})k} \big\|
P_{\leq k} f \big\|_{L^p},  \\
\big\| P_{ k} f \big\|_{L^q} & \lesssim &
2^{(\frac{d}{p}-\frac{d}{q})k} \big\|P_{k} f \big\|_{L^p},
\end{eqnarray*}
where $1\leq p\leq q \leq \infty$.

Next, we record here a refinement of the Sobolev embedding for
radial functions, which will be of use in Section 4.
\begin{lemma}[Radial Sobolev embedding, \cite{TVZ}] Let $d\geq1, \; 1\leq
q\leq\infty,~0<s<d,$ and $\beta\in\R$ obey the conditions
$$\beta>-\tfrac{d}{q},~1\leq\tfrac1{p}+\tfrac1{q'}\leq1+s$$
and the scaling condition
$$\beta+s=\tfrac{d}{p}-\tfrac{d}{q}$$
with at most one of the equalities
$$p=1,~p=\infty,~q=1,~q=\infty,~\tfrac1{p}+\tfrac1{q'}=1+s$$
holding. Then for any spherically symmetric function $f\in\dot
W^{s,p}(\R^d)$, we have
\begin{equation}
\|r^\beta
f\|_{L^q(\R^d)}\lesssim\big\||\nabla|^sf\big\|_{L^p(\R^d)}.
\end{equation}

\end{lemma}
%
%
%
%
%

The Strichartz estimates involve the following definitions:
\begin{definition}[Admissible pairs]\label{def1}
A pair of Lebesgue space exponents $(q,r)$ are called wave
admissible in $\mathbb{R}^{1+d}$, denoted by $(q,r)\in \Lambda_{0}$
when $q,r\geq 2$, and
\begin{equation}\label{equ21}
\tfrac{2}{q}\leq(d-1)\big(\tfrac{1}{2}-\tfrac{1}{r}\big),~\text{and}\quad(q,r,d)\neq
(2,\infty,3).
\end{equation}

\end{definition}

Now we recall the following Strichartz estimates.
\begin{lemma}[Strichartz estimates,\cite{GiV95, KeT98,
LS, Sogge}]\label{strichartzest}  Let $I$ be a compact time interval
and let $u:~I\times\R^d\to\R$ be a solution to the forced wave
equation
$$u_{tt}-\Delta u+F=0$$ with initial data $(u,\dot{u})\big|_{t=t_0}=(u_0,u_1)\in\dot H^s\times\dot
H^{s-1}$ for $s>0$. Then
\begin{equation*}
\|u\|_{L_t^q(I,L_x^r)}\lesssim \big\|\big(u(t_0),\pa_t
u(t_0)\big)\big\|_{\dot{H}^s\times \dot
H^{s-1}}+\|F\|_{L_t^{q_1'}(I, L_x^{r_1'})},
\end{equation*}
where  $(q,r)$ and $(q_1,r_1)$ are admissible pairs satisfying
$$\frac1q+\frac{d}r=\frac{d}2-s=\frac1{q_1'}+\frac{d}{r_1'}-2.$$

\noindent Furthermore, we have the frequency localized Strichartz estimate
\begin{equation}\label{frstrich}
\big\|(\nabla_x,\pa_t)P_ku\big\|_{L_t^{q_1}(I,L_x^{r_1})}\lesssim
2^{k\mu_1}\big\|P_k(u_0,u_1)\big\|_{\dot H^1\times
L^2}+2^{k\mu_2}\|P_kF\|_{L_t^{q_2'}(I; L_x^{r_2'})},
\end{equation}
where $(q_1,r_1)$ and $(q_2,r_2)$  are admissible pairs  and satisfy
$$\mu_1=\delta(r_1)-\tfrac1{q_1},~\mu_2=\mu_1+\delta(r_2)-\tfrac1{q_2},~\delta(r):=d\big(\tfrac12-\tfrac1r\big).$$
\end{lemma}

\subsection{Blow-up for non-positive energies} We recall that in the
case of the focusing equation, any nontrivial solution with
non-positive energy must blow-up in both time directions. Such
result was first proved in Killip, Stovall and Visan \cite{KSV}  for
the solutions to NLW \eqref{waveequ}.  By the same argument as
deriving Theorem 3 in \cite{KSV} with the extended  causality (c.f.
Lemma 2.5, \cite{MZZ})which  plays role of the finite speed propagation,
 we also have

\begin{proposition}[Blow-up for non-positive energies]\label{propbnpe} Let $u(t,x)$
be a solution to problem \eqref{equ1.1} with $\mu=-1$ and with maximal
interval of existence $I_{\max}=(T_-,T_+)$. If $E(u,u_t)\leq0$, then
$(u,u_t)(t)$ is either identically  zero or blows up in finite time
in both time directions, i.e. $T_->-\infty$ and $T_+<+\infty.$
\end{proposition}

\section{Extinction of two scenarios}
 \setcounter{section}{3}\setcounter{equation}{0}

In this section, we preclude two scenarios in the sense of Theorem
\ref{three} under the assumption that Theorem \ref{regular} and
Proposition \ref{propgdzzx} holds. And we will prove Theorem
\ref{regular} and Proposition \ref{propgdzzx} in the next section.

\subsection{High to low frequency cascade}

First, we preclude the high to low frequency cascade solution by
making use of  Theorem \ref{regular}.

\begin{theorem}[No high to low frequency cascade]\label{no frequency-cascades}
Let $d\geq4$. Then there are no radial almost periodic solutions
$u:(T_-,\infty)\times\R^d\to\R$ to problem \eqref{equ1.1} with $N(t)\leq 1$
on $\R^+$ such that
    \begin{equation}\label{rf blowup}
    \|u\|_{S((T_-,\infty))}=+\infty
    \end{equation}
and
    \begin{equation}\label{rf K}
    \varliminf_{t\to\infty}N(t)=0.
    \end{equation}
\end{theorem}

\begin{proof}
 We argue by contradiction. Assume that $u$  were such a solution. Using Theorem \ref{regular},
 we obtain
\begin{equation}
\|(u,u_t)\|_{\dot H^{1}_x\times L^2_x}\lesssim N(t)^\frac13,\quad
\forall~t\in\R^+.
\end{equation}
Combining this inequality with \eqref{rf K}, we get
$$\varliminf_{t\to\infty}\|(u,u_t)\|_{\dot H^{1}_x\times
L^2_x}\lesssim\varliminf_{t\to\infty}N(t)^\frac13=0.$$ On the other
hand, by the H\"older inequality, the Hardy-Littlewood-Sobolev
inequality, the Sobolev embedding and interpolation, we estimate the
potential energy as follows
\begin{align*}
\big\|(|x|^{-3}\ast|u|^2)|u|^2\big\|_{L_x^1}\lesssim
\|u\|_{L_x^\frac{4d}{2d-3}}^4\lesssim\|u\|_{\dot
H^\frac34}^4\lesssim \|u\|_{\dot
H^{1/2}}^{2}\|u\|_{\dot H^1}^{2}\to0,  \;\; \text{as}\; t\to\infty.
\end{align*}
Hence, we have
$$\varliminf_{t\to\infty} E(u,u_t)(t)=0.$$
This together with energy conservation yields
$$E(u_0,u_1)=0.$$
Combining this with Proposition \ref{propbnpe} if $\mu=-1$, we get
$u(t)\equiv0$. This contradicts with \eqref{rf blowup}. And so we
complete the proof of Theorem \ref{no frequency-cascades}.
\end{proof}

\subsection{The soliton-like solution}
Next, we turn to exclude the soliton-like solution under the
assumption that  Proposition \ref{propgdzzx} holds.

\begin{theorem}[No soliton-like solution]\label{no soliton-like solution}
Let $d\geq4$. Then, there are no radial almost periodic solutions
$u:\R\times\R^d\to\R$ to problem \eqref{equ1.1} with $N(t)\equiv 1$ on $\R$
such that
    \begin{equation}\label{equsollkin}
    \|u\|_{S(\R)}=\infty.
    \end{equation}
\end{theorem}

First, we need to establish the following virial identity.

\begin{lemma}[virial identity]\label{vir3drw}
Let $\vec{u}:=(u,u_t)\in(\dot H^1\times L^2)\cap(\dot
H^\frac12\times\dot H^{-\frac12})$ be a solution to problem
\eqref{equ1.1}. Then, for any $R>0$, we have
\begin{equation}\label{equlem1.23dw}
\begin{split}
&\frac{d}{dt}\Big\langle u_t,
\chi_R\big(\tfrac{d-1}2u+ru_r\big)\Big\rangle\\
=&-\frac1{|\mathbb{S}^{d-1}|}E(\vec u)+\int(1-\chi_R)\Big[\frac12|u_r|^2+\frac12|u_t|^2\pm\frac{d-1}2(|x|^{-3}\ast|u|^2)|u|^2\Big]r^{d-1}\,\mathrm{d}r\\
&-\int\Big[\frac12|u_r|^2+\frac12|u_t|^2\Big]\chi_R'r^d\,\mathrm{d}r-\frac{d-1}2\int
uu_r\chi_R'r^{d-1}\,\mathrm{d}r\\
&\pm\int (|x|^{-3}\ast|u|^2)uu_r(1-\chi_R)r^d\,\mathrm{d}r,
\end{split}
\end{equation}
where
$$\langle f,g\rangle\triangleq \int_0^\infty
f(r)g(r)r^{d-1}\,\mathrm{d}r,
$$
and $\chi_R(r)=\chi(r/R),~\chi\in C_c^\infty(\R^+)$ with
$$\chi(r)=\begin{cases} 1,\quad 0\leq r\leq 1,\\
0,\quad r\geq2.\end{cases}$$
\end{lemma}
\begin{proof} We compute  by Leibniz rule
\begin{align}\nonumber
&\frac{d}{dt}\Big\langle u_t,
\chi_R\big(\tfrac{d-1}2u+ru_r\big)\Big\rangle\\\nonumber =&\int
u_{tt}\chi_R\big(\tfrac{d-1}2u+ru_r\big)r^{d-1}\,\mathrm{d}r+\int
u_t\chi_R\big(\tfrac{d-1}2u_t+r\partial_ru_t\big)r^{d-1}\,\mathrm{d}r\\\label{equvirgj}
\triangleq&I_1+I_2.
\end{align}
{\bf The computation of $I_2$:} Using integration by part, we obtain
\begin{align*}
I_2=&\frac{d-1}2\int |u_t|^2\chi_Rr^{d-1}dr+\int u_t\partial_ru_t\chi_Rr^d\,\mathrm{d}r\\
=&\frac{d-1}2\int |u_t|^2\chi_Rr^{d-1}\,\mathrm{d}r-\frac12\int
|u_t|^2\partial_r\big(\chi_Rr^d\big)\,\mathrm{d}r\\
=&-\frac{1}2\int |u_t|^2\chi_Rr^{d-1}\,\mathrm{d}r-\frac12\int
|u_t|^2\chi_R'r^d\,\mathrm{d}r.
\end{align*}
{\bf The computation of $I_1$:} By equation \eqref{equ1.1} and the
radial assumption, we rewrite
$$u_{tt}=\Delta u\mp
(|x|^{-3}\ast|u|^2)u=\partial_{r}^2u+\tfrac{d-1}{r}\partial_ru\mp
 (|x|^{-3}\ast|u|^2)u.$$
 Plugging this equality into $I_1$, one has
\begin{align*}
I_1=&\int
\partial_{r}^2u\chi_R\big(\tfrac{d-1}2u+ru_r\big)r^{d-1}\,\mathrm{d}r+(d-1)\int
\partial_ru\chi_R\big(\tfrac{d-1}2u+ru_r\big)r^{d-2}\,\mathrm{d}r\\
&\mp\int  (|x|^{-3}\ast|u|^2)u\chi_R\big(\tfrac{d-1}2u+ru_r\big)r^{d-1}\,\mathrm{d}r\\
\triangleq&I_{11}+I_{12}\pm I_{13}.
\end{align*}

First, we consider the contribution of $I_{11}+I_{12}$.
Integrating by part, we derive that
\begin{align*}
I_{11}=&\frac{d-1}2\int \partial_{r}^2u
u\chi_Rr^{d-1}\,\mathrm{d}r+\int
\partial_{r}^2uu_r\chi_Rr^d\,\mathrm{d}r\\
=&-\frac{d-1}2\int|u_r|^2\chi_Rr^{d-1}\,\mathrm{d}r-\frac{d-1}2\int
u_ru\partial_r\big(\chi_Rr^{d-1}\big)\,\mathrm{d}r-\frac12\int|u_r|^2\partial_r\big(\chi_Rr^d\big) \,\mathrm{d}r\\
=&-\frac{2d-1}2\int|u_r|^2\chi_Rr^{d-1}\,\mathrm{d}
r+\frac{d-1}4\int|u|^2\partial_r^2\big(\chi_Rr^{d-1}\big)\,\mathrm{d}r-\frac12\int|u_r|^2\chi_R'r^d\,\mathrm{d}r,
\end{align*}
and
\begin{align*}
I_{12}=&\frac{(d-1)^2}2\int
\partial_ruu\chi_Rr^{d-2}\,\mathrm{d}r+(d-1)\int
|u_r|^2\chi_Rr^{d-1}\,\mathrm{d}r\\
=&-\frac{(d-1)^2}4\int
|u|^2\big(\chi_Rr^{d-2}\big)_r\,\mathrm{d}r+(d-1)\int
|u_r|^2\chi_Rr^{d-1}\,\mathrm{d}r.
\end{align*}
Hence,
\begin{align*}
I_{11}+I_{12}=&-\frac{1}2\int|u_r|^2\chi_Rr^{d-1}\,\mathrm{d}r+\frac{d-1}4\int|u|^2\partial_r\big(\chi_R'r^{d-1}\big)\,\mathrm{d}r
-\frac12\int|u_r|^2\chi_R'r^d\,\mathrm{d}r\\
=&-\frac{1}2\int|u_r|^2\chi_Rr^{d-1}\,\mathrm{d}r-\frac{d-1}2\int
uu_r\chi_R'r^{d-1}\,\mathrm{d}r-\frac12\int|u_r|^2\chi_R'r^d\,\mathrm{d}r.
\end{align*}

Second, we turn to consider the contribution of the term $I_{13}$.
Since
$$\iint_{\mathbb{R}^d\times\mathbb{R}^d}\frac{x\cdot(x-y)}{|x-y|^{5}}|u(x)|^2|u(y)|^2dxdy
=-\iint_{\mathbb{R}^d\times\mathbb{R}^d}\frac{y\cdot(x-y)}{|x-y|^{5}}|u(x)|^2|u(y)|^2dxdy,$$
we obtain
\begin{equation*}
\int_{\mathbb{R}^d}(|x|^{-3}*|u|^2)u(x) x\cdot\nabla u(x)
dx=\Big(-\frac{d}{2}+\frac{3}{4}\Big)\iint_{\mathbb{R}^d\times
\mathbb{R}^d}\frac{|u(x)|^2|u(y)|^2}{|x-y|^{3}}dxdy.
\end{equation*}
This symmetrical identity implies
\begin{align*}
I_{13}=&\frac{d-1}2\int \chi_R(|x|^{-3}\ast|u|^2)|u|^2r^{d-1}\,\mathrm{d}r+\int (|x|^{-3}\ast|u|^2)uu_rr^d\,\mathrm{d}r\\
&-\int (|x|^{-3}\ast|u|^2)uu_r(1-\chi_R)r^d\,\mathrm{d}r\\
 =&\frac14\int
(|x|^{-3}\ast|u|^2)|u|^2r^{d-1}\,\mathrm{d}r-\frac{d-1}2\int
(1-\chi_R)(|x|^{-3}\ast|u|^2)|u|^2r^{d-1}\,\mathrm{d}r\\
&-\int (|x|^{-3}\ast|u|^2)uu_r(1-\chi_R)r^d\,\mathrm{d}r.
\end{align*}

Plugging the above computations into \eqref{equvirgj}, we get
\begin{align*}
&\frac{d}{dt}\big\langle u_t, \chi_R(u+ru_r)\big\rangle\\
=&-\int\Big[\frac12|u_r|^2+\frac12|u_t|^2\pm\frac14(|x|^{-3}\ast|u|^2)|u|^2\Big]r^{d-1}\,\mathrm{d}r\\
&+\int\Big[\frac12|u_r|^2+\frac12|u_t|^2\pm\frac{d-1}2(|x|^{-3}\ast|u|^2)|u|^2\Big](1-\chi_R)r^{d-1}\,\mathrm{d}r\\
&-\int\Big[\frac12|u_r|^2+\frac12|u_t|^2\Big]\chi_R'r^d\,\mathrm{d}r-\frac{d-1}2\int
uu_r\chi_R'r^{d-1}\,\mathrm{d}r\\
&\pm\int (|x|^{-3}\ast|u|^2)uu_r(1-\chi_R)r^d\,\mathrm{d}r\\
=&-\frac1{|\mathbb{S}^{d-1}|}E(\vec u)+\int(1-\chi_R)\Big[\frac12|u_r|^2+\frac12|u_t|^2\pm\frac{d-1}2(|x|^{-3}\ast|u|^2)|u|^2\Big]r^{d-1}\,\mathrm{d}r\\
&-\int\Big[\frac12|u_r|^2+\frac12|u_t|^2\Big]\chi_R'r^d\,\mathrm{d}r-\frac{d-1}2\int
uu_r\chi_R'r^{d-1}\,\mathrm{d}r\\
&\pm\int (|x|^{-3}\ast|u|^2)uu_r(1-\chi_R)r^d\,\mathrm{d}r.
\end{align*}

\noindent Thus, we complete the proof of this lemma.
\end{proof}

\vskip 0.2in

\begin{proof}[The proof of Theorem \ref{no soliton-like solution}]\;
 We argue by contradiction. Assume that $u$ were such a solution.
 We claim that
\begin{equation}\label{claimvir}
\forall~\eta>0,~E(u,u_t)\leq C\eta^\frac12.
\end{equation}
Combining this claim with Proposition \ref{propbnpe} if $\mu=-1$, we
get $u(t)\equiv0$. This contradicts with \eqref{equsollkin}. Thus,
we preclude the soliton solution in the sense of Theorem \ref{three}
under the assumption that  Proposition \ref{propgdzzx} holds.\end{proof}

\vskip 0.2in

Next, we use Proposition \ref{propgdzzx} and Lemma \ref{vir3drw} to
show the claim \eqref{claimvir}.
\vskip0.1cm

{\bf $\bullet$ The proof of claim \eqref{claimvir}:} By  Proposition
\ref{propgdzzx}, we deduce that for any $\eta>0$, there exists a
$R_0=R_0(\eta)>0$ such that for any $R\geq R_0(\eta)$
\begin{equation}\label{compatequ}
\int_R^{+\infty}\big[u_t^2+u_r^2\big](t)r^{d-1}\,\mathrm{d}r<\eta.
\end{equation}
This inequality together with Sobolev embedding
$$\dot H^{1/2}(\R^d)\cap\dot
H^1(\R^d)\hookrightarrow
L_x^\frac{4d}{2d-3}(\R^d)$$
yields
\begin{equation}
\int_R^{+\infty} |u(t)|^\frac{4d}{2d-3}r^{d-1}\,\mathrm{d}r<\eta.
\end{equation}
Hence, using the H\"older inequality and the Hardy-Littlewood-Sobolev
inequality, we obtain
\begin{equation}
\left\{\begin{aligned}  &\Big|\int_0^{+\infty}
(1-\chi_R)\Big(\frac12u_t^2+\frac12u_r^2\pm\frac{d-1}2(|x|^{-3}\ast|u|^2)|u|^2\Big)r^{d-1}\,\mathrm{d}r\Big|\leq
C\eta\\
&\Big|\int_0^{+\infty}\Big(\frac12u_t^2+\frac12u_r^2\Big)r\chi_R'
r^{d-1}\,\mathrm{d}r \Big|\lesssim \eta.
\end{aligned}\right.
\end{equation}
On the other hand, by H\"older's inequality, Hardy's inequality, the
radial Sobolev inequality and Proposition \ref{propgdzzx}, we get
 \begin{align*}
\Big|\int_0^{+\infty}uu_rr\chi_{R}'r^{d-2}\;dr\Big|\lesssim&\Big(\int_{R}^{2R}u_r^{2}
r^{d-1}\;dr\Big)^{\frac12}\Big(\int_{R}^{2R}u^{2}
r^{d-3}\;dr\Big)^{\frac12}\\
\lesssim&\eta^{\frac12}\Big(\int_{\R^{d}}\frac{u^{2}}{|x|^{2}}\;dx\Big)^{\frac{1}{2}}\\
\lesssim&\eta^{\frac12}\|u\|_{\dot H^1_x(\R^d)}\lesssim
\eta^{\frac12},
\end{align*}
and
\begin{align*}
\Big|\int_0^{+\infty}(|x|^{-3}\ast|u|^{2})uu_r(1-\chi_{R})r^d\;dr\Big|
\lesssim&\||x|^{-3}\ast|u|^{2}\|_{L^{\frac{d}{2}}_x(\R^d)}\Big(\int_R^{+\infty}u_r^{2}r^{d-1}
\;dr\Big)^{\frac{1}{2}}\big\||x|u\big\|_{L^{\frac{2d}{d-4}}_x(\R^d)}\\
\lesssim&\|u\|_{\dot H^{\frac12}_x(\R^d)}^{2}\eta^{\frac12}\|\nabla u\|_{L^{2}_x(\R^d)}\\
\lesssim&\eta^{\frac12}.
\end{align*}

\noindent Integrating the inequality \eqref{equlem1.23dw} from  $0$ to
$T$ in time and setting $R=T\gg R_0(\eta)$, one has
\begin{align*}
E(u,u_t)\leq& C\eta^\frac12+\frac1T\Big|\big\langle
u_t(t),~\chi_R\big(\tfrac{d-1}2u+ru_r\big)(t)\big\rangle\big|_0^T\Big|\\
\leq&C\eta^\frac12+\frac{C}T\int_0^T|u_t(T)|\cdot|u(T)|r^{d-1}dr+\frac{C}{T}\int_0^T|u_t(0)|\cdot|u(0)|r^{d-1}\,\mathrm{d}r\\
&+\frac{C}T\int_0^T|u_t(T)|\cdot|u_r(T)|r^ddr+\frac{C}T\int_0^T|u_t(0)|\cdot|u_r(0)|r^d\,\mathrm{d}r.
\end{align*}
On the other hand, by the H\"older inequality, the Sobolev embedding and
\eqref{compatequ}, we get
\begin{align*}
\frac{1}T\int_0^T|u_t|\cdot|u|r^{d-1}
\,\mathrm{d}r\leq&\frac1T\Big(\int_0^T|u_t|^2r^{d-1}\,\mathrm{d}r\Big)^\frac12\Big(\int_0^T|u|^\frac{2d}{d-1}r^{d-1}\,\mathrm{d}r\Big)^\frac{d-1}{2d}
\Big(\int_0^Tr^{d-1}\,\mathrm{d}r\Big)^\frac1{2d}
\\
\leq&\frac{C}{T^\frac12}\|u_t\|_{L_x^2(\R^d)}\|u\|_{\dot
H^\frac12_x(\R^d)},
\end{align*}
and
\begin{align*}
\frac{1}T\int_0^T|u_t|\cdot|u_r|r^d\,\mathrm{d}r
\leq&\frac{R_0(\eta)}T\int_0^{R_0(\eta)}|u_t|\cdot|u_r|r^{d-1}\,\mathrm{d}r+\frac{1}T\int_{R_0(\eta)}^T|u_t|\cdot|u_r|r^d\,\mathrm{d}r\\
\leq&\frac{R_0(\eta)}T\|u_t\|_2\|u\|_{\dot
H^1}+\Big(\int_{R_0(\eta)}^{+\infty}
u_t^2r^{d-1}\,\mathrm{d}r\Big)^\frac12\Big(\int_{R_0(\eta)}^{+\infty}
u_r^2r^{d-1}\,\mathrm{d}r\Big)^\frac12\\
\leq&\frac{R_0(\eta)}T+\eta.
\end{align*}
Thus,
$$E(u_0,u_1)\leq  C\eta^\frac12+\frac{C}{T^\frac12}+C\frac{R_0(\eta)}T.$$
Letting $T\to+\infty$, we obtain the claim \eqref{claimvir}.

\vskip 0.1cm

In sum, it suffices to prove Theorem \ref{regular} and Proposition
\ref{propgdzzx} which will be shown in the next section.

\section{Additional regularity}
 \setcounter{section}{4}\setcounter{equation}{0}

As stated in Section 3, it remains to show Theorem \ref{regular} and
Proposition \ref{propgdzzx}. More precisely, we need to prove that
the
 solutions which are radial almost periodic modulo symmetries enjoy
the additional regularity.

\subsection{Proof of  Theorem \ref{regular}}
In this subsection, we  will divide two steps to prove  Theorem
\ref{regular}. First, we show that the almost periodic solutions in
Theorem \ref{three} lie in $\dot H^\frac23(\R^d)\times\dot
H^{-\frac13}(\R^d)$ when $t\in\R^+$ by bootstrap argument. And then
we utilize this result to show  Theorem \ref{regular}.

\begin{theorem}\label{thmpartial}
Let $u:(T_-,+\infty)\times\R^d\to\R$ be a radial solution of problem \eqref{equ1.1} which is
almost periodic modulo symmetries in the sense of Theorem
\ref{three}. Then for any $t_0\in \R^+$, we have
\begin{equation}\label{equbfzzx3dwr}
\big\| (u,u_t)(t_0)\big\|_{\dot H^\frac23_x\times\dot
H^{-\frac{1}3}_x}\lesssim N(t_0)^{\frac16}.
\end{equation}
\end{theorem}

Before showing this theorem, we need the refined Strichartz-type
estimate in a  small time interval.

\begin{lemma}[Refined Strichartz-type
estimate]\label{gaijstes3drw} Let $u$ be as in Theorem
\ref{thmpartial} and
$J_\delta(t_0):=\big[t_0-\tfrac\delta{N(t_0)},t_0+\tfrac\delta{N(t_0)}\big]$.
Then, for any $\eta>0$, there exists a $\delta=\delta(\eta)>0$ such
that for all $t_0\in \R^+$
\begin{equation}\label{refse3drw}
\|u\|_{S(J_\delta(t_0))}\lesssim\eta.
\end{equation}

\end{lemma}

\begin{proof} Without loss of generality, we assume that $t_0=0$. By
Duhamel formula, we get
\begin{equation}\label{equ4.1duh}
\|u\|_{S(J_\delta(0))}\lesssim\big\|L(t)(u_0,u_1)\big\|_{S(J_\delta(0))}+\Big\|\int_0^tL(t-s)(0,\pm
(|x|^{-3}\ast|u|^2)u)\,\mathrm{d}s\Big\|_{S(J_\delta(0))},
\end{equation}
where $L(t)(f,g):=\dot K(t)f+K(t)g$ and $K(t)$ is defined by
\eqref{v0tdefin}.
 For the free part,
using Strichartz estimate and \eqref{apss}, we have
\begin{align}
\big\|L(t)P_{\geq
\log_2(C(\eta)N(0))}(u_0,u_1)\big\|_{S(J_\delta(0))}\lesssim&\big\|P_{\geq
\log_2(C(\eta)N(0))}(u_0,u_1)\big\|_{\dot H^{1/2}\times\dot
H^{-1/2}}\leq \eta.
\end{align}
On the other hand, we obtain by the Bernstein inequality,
\begin{equation*}
\big\|L(t)P_{\leq
\log_2(C(\eta)N(0))}(u_0,u_1)\big\|_{L_x^q}\lesssim
\big(C(\eta)N(0)\big)^\frac1q\big\|(u_0,u_1)\big\|_{\dot
H^{1/2}\times\dot H^{-1/2}},\quad q=\tfrac{2(d+1)}{d-1}.
\end{equation*}
Integrating this inequality in time, we obtain
\begin{equation}
\big\|L(t)P_{\leq
\log_2(C(\eta)N(0))}(u_0,u_1)\big\|_{S(J_\delta(0))}\lesssim
C(\eta)^\frac1q\delta^\frac1q.
\end{equation}

For the inhomogenous part, by the Strichartz estimate,  the H\"older
inequality and the Hardy-Littlewood-Sobolev inequality, we have
\begin{equation*}
\begin{split}
&\Big\|\int_0^tL(t-s)(0,\pm
(|x|^{-3}\ast|u|^2)u)\,\mathrm{d}s\Big\|_{S(J_\delta(0))}\\
\lesssim&\big\|(|x|^{-3}\ast|u|^2)u\big\|_{L_t^\frac{q}2L_x^{r_1'}(J_\delta(0)\times\R^d)}\qquad \big(q=\tfrac{2(d+1)}{d-1},
r_1=\tfrac{2d(d+1)}{d^2-5}\big)\\
\lesssim&\|u\|_{L_t^\infty\dot H^{\frac12}}
\|u\|_{S(J_\delta(0))}^2.
\end{split}
\end{equation*}

\noindent Plugging the above estimates into \eqref{equ4.1duh}, we deduce that
\begin{equation*}
\|u\|_{S(J_\delta(0))}\lesssim\eta+C(\eta)^\frac1q\delta^\frac1q+\|u\|_{S(J_\delta(0))}^2.
\end{equation*}
Hence, we obtain by the bootstrap argument
$$\|u\|_{S(J_\delta(0))}\lesssim\eta.$$
This ends the proof of Lemma \ref{gaijstes3drw}.
\end{proof}

\vskip 0.2in

 Now we use the above lemma to show Theorem
\ref{thmpartial}.

\vskip 0.2in

 {\bf Proof of  Theorem \ref{thmpartial}.} By translation invariance
 in time, we may assume $t_0=0$. Then, we aim to prove
\begin{equation}
 \| (u,u_t)(0)\|_{\dot H^\frac{2}3\times\dot
H^{-\frac13}}\lesssim N(0)^\frac16.
\end{equation}
Since
 $$\| (u,u_t)(0)\|_{\dot H^\frac{2}3\times\dot
H^{-\frac13}}^2\simeq \sum_{k\in\Z}\Big\|P_k(u,u_t)(0)\Big\|_{\dot
H^\frac23\times\dot H^{-\frac13}}^2,$$
we only need to construct a frequency envelope
$\{\al_k(0)\}_{k\in\Z}$ satisfying
\begin{equation}\label{bfzzxgj3drw}
\left\{\begin{aligned}
&\Big\|P_k(u,u_t)(0)\Big\|_{\dot H^\frac23\times\dot
H^{-\frac13}}\lesssim 2^{\frac16k}\alpha_k(0),\\
&\Big\|\big\{2^{\frac16k}\alpha_k(0)\big\}_{k\in\Z}\Big\|_{\ell^2}\lesssim
N(0)^\frac16.
\end{aligned}\right.
\end{equation}

\begin{proposition}\label{prop2.23drw}
Let $\eta>0$ be a small constant and let $J_\delta(0)$  be as in
Lemma \ref{gaijstes3drw}. Define
\begin{equation}
\left\{\begin{aligned}
& a_k:=2^{\frac{k}2}\|P_ku\|_{L_t^\infty
(J_\delta(0),L_x^2)}+2^{-\frac{k}2}\|P_ku_t\|_{L_t^\infty
(J_\delta(0),L_x^2)}+2^{\frac{k}4}\|P_ku\|_{L_t^{q_1}(J_\delta(0),L_x^{r_1})}\\
&a_k(0):=2^{\frac{k}2}\|P_ku(0)\|_{
L_x^2}+2^{-\frac{k}2}\|P_ku_t(0)\|_{L_x^2},
\end{aligned}\right.
\end{equation}
where
$\frac1{q_1}+\frac{d}{r_1}=\frac{d}2-\frac14,~q_1=\frac{2d+1}{d-1}$,
and
\begin{equation}
\alpha_k:=\sum_j2^{-\frac15|j-k|}a_j,\quad
\alpha_k(0):=\sum_j2^{-\frac15|j-k|}a_j(0).
\end{equation}
It follows from the definition that
\begin{equation}\label{ydyxj3drw}
\Big\|P_k(u,u_t)(0)\Big\|_{\dot H^\frac23\times\dot
H^{-\frac13}}\simeq 2^{\frac{k}6}\Big\|P_k(u,u_t)(0)\Big\|_{\dot
H^{\frac12}\times\dot H^{-\frac12}}
\simeq  2^{\frac{k}6}a_k(0)
\lesssim 2^{\frac{k}6}\alpha_k(0).
\end{equation}
Then, we have
\begin{align}\label{equ2.83drw}
a_k\lesssim& a_k(0)+\eta^2\sum_{j\geq
k-3}2^{\frac{k-j}4}a_j\\\label{equ2.93drw}
\alpha_k\lesssim&\alpha_k(0).
\end{align}
\end{proposition}

\begin{proof} We first use the frequency localized Strichartz estimate
\eqref{frstrich} to get
\begin{align*}
a_k\lesssim&2^{\frac{k}2}\|P_ku(0)\|_{
L_x^2}+2^{-\frac{k}2}\|P_ku_t(0)\|_{L_x^2}+2^{\frac{k}4}\Big\|P_k\big[(|x|^{-3}\ast|u|^2)u\big]\Big\|_{L_t^{q_2'}L_x^{r_2'}(J_\delta(0)\times\R^d)},
\end{align*}
where
$q_2=\tfrac{(2d+1)(d+1)}{4d+3},~\delta(r_2)-\tfrac1{q_2}=\tfrac34.$
Observing that
$$P_k\big[(|x|^{-3}\ast|P_{\leq k-4}u|^2)P_{\leq
k-4}u\big]=0,$$
  and
$$\big\|
(|x|^{-3}\ast|u|^2)P_ju\big\|_{L_t^{q_2'}L_x^{r_2'}}+\big\|
(|x|^{-3}\ast(u\cdot
P_ju)u\big\|_{L_t^{q_2'}L_x^{r_2'}}\lesssim\|P_ju\|_{L_t^{q_1}L_x^{r_1}}\|u\|_{S(J_\delta(0))}^2\lesssim
\eta^22^{-\frac{j}4}a_j,$$
we obtain by \eqref{refse3drw}
\begin{equation*}
a_k\lesssim a_k(0)+\eta^2\sum_{j\geq k-3}2^{\frac{k-j}4}a_j,
\end{equation*}
as desired estimate in \eqref{equ2.83drw}.

By \eqref{equ2.83drw}, we have
$$\alpha_k=\sum_j2^{-\frac{|j-k|}5}a_j\lesssim
\sum_ja_j(0)2^{-\frac{|j-k|}5}+\eta^2\sum_j2^{-\frac{|j-k|}5}\sum_{j_1\geq
j-3}2^{\frac{j-j_1}4}a_{j_1}.$$ This inequality together with
\begin{equation}\label{jishuqhgj}
\begin{cases}
\sum\limits_{j_1\leq k}\sum\limits_{j\leq
j_1+3}2^{\frac{j-j_1}4}2^{\frac{j-k}5}a_{j_1}\lesssim\sum\limits_{j_1\leq
k}2^{\frac{j_1-k}5}a_{j_1}\lesssim \alpha_k\\
\sum\limits_{j_1>k}\sum\limits_{j\leq
j_1+3}2^{\frac{j-j_1}4}2^{-\frac{|j-k|}5}a_{j_1}\lesssim\sum\limits_{j_1>k}\big[2^{-\frac{j_1-k}4}
+2^{-\frac{j_1-k}5}\big]a_{j_1}\lesssim\alpha_k
\end{cases}
\end{equation}
yields
\begin{equation}\label{equ:alphak}
\alpha_k\lesssim \alpha_k(0)+\eta^2\alpha_k.\end{equation} On the
other hand, by constructions, we know that $\alpha_k(0)$ is
uniformly bounded in $k$.
 And so \eqref{equ2.93drw} follows by \eqref{equ:alphak}. We
conclude Proposition \ref{prop2.23drw}.
\end{proof}

\begin{remark}
From the proof of Proposition
\ref{prop2.23drw}, we also obtain
\begin{equation}\label{jybdgj3drw}
2^{\frac{k}2}\Big\|P_k\int_0^{\frac\delta{N(0)}}\frac{e^{-it\sqrt{-\Delta}}}{\sqrt{-\Delta}}(|x|^{-3}\ast|u|^2)u(t)\,\mathrm{d}t\Big\|_{L_x^2}\lesssim
\eta^2\sum_{j\geq k-3}2^{\frac{k-j}4}a_j.
\end{equation}
\end{remark}

Now we turn to estimate the term in large time interval. Define
\begin{equation}\label{dchi}
\chi(x)\in C_c^\infty(\R^d)\qquad \chi(x)=\begin{cases}1\quad
|x|\leq
1\\
0\quad |x|\geq2.
\end{cases}.
\end{equation}

\begin{lemma}[Large time]\label{lem2.33drw} For any
$s_0\in(0,1/2)$, we have
\begin{equation}\label{equlartime}
\int_{\delta/N(0)}^\infty\Big\|\frac{e^{-it\sqrt{-\Delta}}}{\sqrt{-\Delta}}(1-\chi)\Big(\frac{8x}{|t|}\Big)(|x|^{-3}\ast|u|^2)u\Big\|_{\dot
H^{\frac12+s_0}}\,\mathrm{d}t\lesssim N(0)^{s_0}\delta^{-s_0}.
\end{equation}
In particularly, there exists a sequence $\{b_k\}_{k\in\Z}$ such
that
\begin{equation}\label{xqbjxld}
\Big\|P_k\Big(\int_{\delta/N(0)}^{\infty}\frac{e^{-i\tau\sqrt{-\Delta}}}{\sqrt{-\Delta}}(1-\chi)\Big(\frac{8x}{|\tau|}\Big)(|x|^{-3}\ast|u|^2)u
\,\mathrm{d}\tau\Big)\Big\|_{\dot H^{\frac12}}\lesssim
2^{-\frac{k}6}b_k,
\end{equation}
and
\begin{equation}\label{gjxgjzzx3drwewgj}
\big\|\{b_k\}_{k\in\Z}\big\|_{\ell^2}\lesssim
 N(0)^\frac16.
\end{equation}
\end{lemma}

\begin{proof} First, we use  the Sobolev embedding to get
\begin{align}\label{equlargte}
&\Big\|\frac{e^{-it\sqrt{-\Delta}}}{\sqrt{-\Delta}}(1-\chi)\Big(\frac{8x}{|t|}\Big)(|x|^{-3}\ast|u|^2)u\Big\|_{\dot
H^{\frac12+s_0}}\\\nonumber
=&\Big\|(1-\chi)\Big(\frac{8x}{|t|}\Big)(|x|^{-3}\ast|u|^2)u\Big\|_{\dot
H^{-\frac12+s_0}}\\\nonumber
\lesssim&\Big\|(1-\chi)\Big(\frac{8x}{|t|}\Big)(|x|^{-3}\ast|u|^2)u\Big\|_{L^p}\qquad
\big(\tfrac{d}p=\tfrac{d}2+\tfrac12-s_0\big).
\end{align}
Next, by the  H\"older inequality and the radial Sobolev
embedding, we obtain
\begin{align*}
\Big\|(1-\chi)\Big(\frac{8x}{|t|}\Big)(|x|^{-3}\ast|u|^2)u\Big\|_{L^p}\lesssim&
\Big\|(1-\chi)\Big(\frac{8x}{|t|}\Big)u\Big\|_{L^q}\big\||x|^{-3}\ast|u|^2\big\|_{L^\frac{d}2}\\
\lesssim&|t|^{-1-s_0}\big\||x|^{1+s_0}u\big\|_{L^q}\|u\|_{L^\frac{2d}{d-1}}^2\\
\lesssim&|t|^{-1-s_0}\|u\|_{\dot H^{\frac12}}^3,
\end{align*} where $q=\frac{2d}{d-3-2s_0},~\frac1p=\frac1q+\frac2d$.
Plugging this inequality into \eqref{equlargte} yields
\eqref{equlartime}.

The inequality \eqref{xqbjxld} follows by taking
 \begin{equation*}
 s_0(k)=\begin{cases}
 \frac1{4},\quad 2^k\geq N(0)\\
 \frac1{8},\quad 2^k<N(0)
 \end{cases}\qquad b_k=2^{-k(s_0(k)-\frac16)}N(0)^{s_0(k)},
 \end{equation*}
in \eqref{equlartime}.
\end{proof}

\vskip 0.2in

Since the first inequality in \eqref{bfzzxgj3drw} was already
established in \eqref{ydyxj3drw}, it remains to prove the second
inequality in \eqref{bfzzxgj3drw}. By the definition of $\al_k(0)$,
we only need to estimate $a_k(0)$. For this purpose, we denote
\begin{equation}\label{defvt}
v=u+\frac{i}{\sqrt{-\Delta}}u_t,
\end{equation}
then
 $$\|v\|_{\dot H^1}\simeq\| (u,u_t)(t)\|_{\dot H^1\times
L^2},\;\text{and}\;  a_k(0)=\|P_kv(0)\|_{\dot H^\frac12}.$$
 Since
$u_{tt}-\Delta u=\pm (|x|^{-3}\ast|u|^2)u$, $v$ satisfies
$$v_t=u_t+\frac{i}{\sqrt{-\Delta}}u_{tt}=u_t+\frac{i}{\sqrt{-\Delta}}\big[\Delta
u\pm (|x|^{-3}\ast|u|^2)u\big]=-i\sqrt{-\Delta} v\pm
\frac{i}{\sqrt{-\Delta}}(|x|^{-3}\ast|u|^2)u.$$ Hence, for
$T\in(T_-,0)$, we have
$$v(0)=e^{iT\sqrt{-\Delta}}v(T)\pm\frac{i}{\sqrt{-\Delta}} \int_T^0
e^{i\tau\sqrt{-\Delta}}(|x|^{-3}\ast|u|^2)u(\tau)\,\mathrm{d}\tau.$$

 Fixing $T_1\in(T_-,0),$ and using both the Duhamel formula and  no-waste Duhamel formula,
 we have
\begin{align}\nonumber
&a_k(0)^2=\big\|P_kv(0)\big\|_{\dot H^{\frac12}}^2=\big\langle
P_kv(0),P_kv(0)\big\rangle_{\dot H^\frac12}\\\nonumber =&\Big\langle
P_k\big(e^{iT_1\sqrt{-\Delta}}v(T_1)\big),P_kv(0)\Big\rangle_{\dot
H^{\frac12}}\\\nonumber &+\lim_{T_2\to \infty}\Big\langle
\frac{i}{\sqrt{-\Delta}}P_k\Big(\int_{T_1}^0e^{it\sqrt{-\Delta}}(|x|^{-3}\ast|u|^2)u(t)\,\mathrm{d}t\Big),\frac{\pm
i}{\sqrt{-\Delta}}P_k\int_0^{T_2}e^{i\tau\sqrt{-\Delta}}(|x|^{-3}\ast|u|^2)u(\tau)\,\mathrm{d}\tau\Big\rangle_{\dot
H^{\frac12}}\\\label{equpkvgj} \triangleq&I_1+\lim_{T_2\to
\infty}I_2,
\end{align}
where $\big\langle f,g\big\rangle_{\dot H^\frac12}=\int
|\nabla|^\frac12f|\nabla|^\frac12\bar{g}\,\mathrm{d}x.$
\vskip0.1cm

{\bf $\bullet$ The estimate of $I_1$:} We have by the no-waste
Duhamel's formula
\begin{align}\label{equi1gj}
\lim_{T_1\to T_-}I_1 =&\lim_{T_1\to T_-}\Big\langle
P_k\big(e^{iT_1\sqrt{-\Delta}}v(T_1)\big),P_kv(0)\Big\rangle_{\dot
H^{\frac12}}=0.
\end{align}

{\bf $\bullet$ The estimate of $I_2$:} Define
$$A:=P_k\Big(\int_{-\delta/N(0)}^0\frac{ie^{it\sqrt{-\Delta}}}{\sqrt{-\Delta}}(|x|^{-3}\ast|u|^2)u(t)\,\mathrm{d}t+
\int_{T_1}^{-\delta/N(0)}
\frac{ie^{it\sqrt{-\Delta}}}{\sqrt{-\Delta}}(1-\chi)\Big(\frac{8x}{|t|}\Big)(|x|^{-3}\ast|u|^2)u\,\mathrm{d}t\Big)$$
and
$$B:=P_k\Big( \int_{T_1}^{-\delta/N(0)}
\frac{ie^{it\sqrt{-\Delta}}}{\sqrt{-\Delta}}\chi\Big(\frac{8x}{|t|}\Big)(|x|^{-3}\ast|u|^2)u\,\mathrm{d}t\Big).$$
Similarly,
$$A':=P_k\Big(\int_0^{\delta/N(0)}\frac{ie^{i\tau\sqrt{-\Delta}}}{\sqrt{-\Delta}}(|x|^{-3}\ast|u|^2)u(\tau)\,\mathrm{d}\tau+
\int_{\delta/N(0)}^{T_2}
\frac{ie^{i\tau\sqrt{-\Delta}}}{\sqrt{-\Delta}}(1-\chi)\Big(\frac{8x}{|\tau|}\Big)(|x|^{-3}\ast|u|^2)u\,\mathrm{d}\tau\Big)$$
and
$$B':=P_k\Big( \int_{\delta/N(0)}^{T_2}
\frac{ie^{i\tau\sqrt{-\Delta}}}{\sqrt{-\Delta}}\chi\Big(\frac{8x}{|\tau|}\Big)(|x|^{-3}\ast|u|^2)u\,\mathrm{d}\tau\Big).$$
Then, $I_2$ can be expressed by
\begin{align*}
I_2=\big\langle
A+B,A'\big\rangle_{\dot H^\frac12}+\big\langle
A,A'+B'\big\rangle_{\dot H^\frac12}-\big\langle
A,A'\big\rangle_{\dot H^\frac12}+\big\langle B,B'\big\rangle_{\dot
H^\frac12}.
\end{align*}

First, we estimate the contribution of the term  $\big\langle
A+B,A'\big\rangle_{\dot H^\frac12}$. Using Lemma \ref{lem2.33drw}
and \eqref{jybdgj3drw}, we get
$$\|A'\|_{\dot H^\frac12}\lesssim \eta^2\sum_{j\geq
k-3}2^{\frac{k-j}4}a_j+2^{-\frac{k}6}b_k.$$ On the other hand,
observing that
$$\lim_{T_1\to T_-}\frac1{\sqrt{-\Delta}}\int_{T_1}^0e^{it\sqrt{-\Delta}}(|x|^{-3}\ast|u|^2)u(t)\,\mathrm{d}t\rightharpoonup
v(0)\quad \text{in}\quad \dot H^{\frac12},$$  we obtain
\begin{align}\nonumber
&\lim_{T_1\to T_-}\lim_{T_2\to \infty}\Big|\big\langle
A+B,A'\big\rangle_{\dot H^{\frac12}}\Big|\\\nonumber =&\lim_{T_1\to
T_-}\lim_{T_2\to \infty}\Big\langle
P_k\Big(\int_{T_1}^0\frac{e^{it\sqrt{-\Delta}}}{\sqrt{-\Delta}}(|x|^{-3}\ast|u|^2)u(t)\,\mathrm{d}t\Big),A'\Big\rangle_{\dot
H^\frac12}\\\nonumber =&\big\langle
P_kv(0),A'(+\infty)\big\rangle_{\dot H^{\frac12}}\\\nonumber
\lesssim&\|P_kv(0)\|_{\dot H^{\frac12}}\|A'(+\infty)\|_{\dot
H^{\frac12}}\\\label{diyabapgj}
\lesssim&a_k(0)\big(\eta^2\sum_{j\geq
k-3}2^{\frac{k-j}4}a_j+2^{-\frac{k}6}b_k\big),
\end{align} where $A'(+\infty)=\lim\limits_{T_2\to+\infty}A'$.
By symmetry, we also have
\begin{equation}
\lim_{T_1\to T_-}\lim_{T_2\to \infty}\Big|\big\langle
A,A'+B'\big\rangle_{\dot H^{\frac12}}\Big|\lesssim
a_k(0)\big(\eta^2\sum_{j\geq
k-3}2^{\frac{k-j}4}a_j+2^{-\frac{k}6}b_k\big).
\end{equation}

Next, we consider the contribution of the term  $\big\langle
A,A'\big\rangle_{\dot H^{\frac12}}$. By \eqref{jybdgj3drw} and
\eqref{gjxgjzzx3drwewgj}, we get
\begin{equation}\label{diyaapgj}
\Big|\big\langle A,A'\big\rangle_{\dot
H^{\frac12}}\Big|\lesssim\|A\|_{\dot H^{\frac12}}\|A'\|_{\dot
H^{\frac12}}\lesssim \big(\eta^2\sum\limits_{j\geq
k-3}2^{\frac{k-j}4}a_j\big)^2+2^{-\frac{k}3}b_k^2.
\end{equation}

Finally, we turn to estimate the contribution of the term
$\big\langle B,B'\big\rangle_{\dot H^{\frac12}}$. We rewrite
$\big\langle B,B'\big\rangle_{\dot H^{\frac12}}$ as
$$\int_{T_1}^{-\frac\delta{N(0)}}\int_{\frac\delta{N(0)}}^{T_2}\Big\langle
\chi\Big(\frac{8x}{|t|}\Big)(|x|^{-3}\ast|u|^2)u, P_k^2\Big(\frac{
e^{i(\tau-t)\sqrt{-\Delta}}}{(-\Delta)^{\frac12}}\chi\Big(\frac{8x}{|\tau|}\Big)(|x|^{-3}\ast|u|^2)u\Big)\Big\rangle_{L_x^2}
\,\mathrm{d}t\,\mathrm{d}\tau.$$ The kernel of $P_k^2\frac{
e^{i(\tau-t)\sqrt{-\Delta}}}{(-\Delta)^{\frac12}}$ is given by
\begin{align*}
K_k(x)=K_k(|x|)=&c\int_{\R^d}e^{ix\cdot\xi}\phi\Big(\frac{|\xi|}{2^k}\Big)^2e^{i(\tau-t)|\xi|}|\xi|^{-1}\,\mathrm{d}\xi\\
=&c\int_0^\infty\phi\Big(\frac{\rho}{2^k}\Big)^2e^{i(\tau-t)\rho}\rho^{-1}\rho^{d-1}\int_{\mathbb S^{d-1}}e^{i\rho|x|\theta\cdot w}\,\mathrm{d}
\theta \,\mathrm{d}\rho\\
=&c2^{k(d-1)}\int_{\mathbb
S^{d-1}}\int_0^\infty\phi(\rho)^2e^{i2^k\rho\big[(\tau-t)+|x|\theta\cdot
w\big]}\rho^{d-2}\,\mathrm{d}\rho\,\mathrm{d}\theta ,
\end{align*}
where $x=|x|w,~\xi=\rho\theta,~w,\theta\in \mathbb S^{d-1}.$ By the
stationary phase argument, we obtain for $L\in\mathbb N$
\begin{equation}\label{est:kkxy}
\begin{split}
\big|K_k(x-y)\big|\lesssim&\frac{2^{(d-1)k}}{\big\langle2^k\big|(\tau-t)-|x-y|\big|\big\rangle^{L}}\\
\lesssim&2^{(d-1)k}\big\langle2^k|\tau-t|\big\rangle^{-L}\\
\lesssim&
2^{(d-1)k}\big\langle2^k|\tau|\big\rangle^{-L/2}\big\langle2^k|t|\big\rangle^{-L/2},
\end{split}
\end{equation} where we have applied the support property
$$|x|\leq\frac{|t|}4,~|y|\leq\frac{|\tau|}4,\; \;\text{and}\;\;
\tau>\frac\delta{N(0)},~t<-\frac\delta{N(0)}$$
to get
$$|x-y|\leq\frac{|t-\tau|}4\Longrightarrow
(\tau-t)-|x-y|\geq\frac12|\tau-t|,$$ and
$$\big\langle2^k|\tau-t|\big\rangle^2\geq\big\langle2^k|\tau|\big\rangle\big\langle2^k|t|\big\rangle.$$
By  the H\"older inequality, the Hardy-Littlewood-Sobolev inequality
and the Sobolev embedding, we estimate by choosing $L=d$

\begin{align}\nonumber
&\big\langle B,B'\big\rangle_{\dot H^{\frac12}}\\\nonumber
\lesssim&\int_{T_1}^{-\frac\delta{N(0)}}\int_{\frac\delta{N(0)}}^{T_2}\Big|\Big\langle\chi\Big(\frac{8x}{|t|}\Big)(|x|^{-3}\ast|u|^2)
u, K_k\ast \big[(|x|^{-3}\ast|u|^2)\chi\Big(\frac{8x}{|\tau|}\Big)
u\big]\Big\rangle_{L^2}\Big|\,\mathrm{d}t
\,\mathrm{d}\tau\\\nonumber
\lesssim&\int_{T_1}^{-\frac\delta{N(0)}}\int_{\frac\delta{N(0)}}^{T_2}\big\|(|x|^{-3}\ast|u|^2)\chi\Big(\frac{8x}{|t|}\Big)
u\big\|_{L_x^1}\big\|(|x|^{-3}\ast|u|^2)\chi\Big(\frac{8x}{|\tau|}\Big)
u\big\|_{L_x^1}2^{(d-1)k}\big\langle2^k|\tau-t|\big\rangle^{-d}\,\mathrm{d}t\,\mathrm{d}\tau\\\nonumber
\lesssim&\big\||x|^{-3}\ast|u|^2\big\|_{L_t^\infty
L_x^\frac{d}2}^2\|u\|_{L_x^\frac{2d}{d-1}}^2
\int_{-\infty}^{-\frac\delta{N(0)}}\int_{\frac\delta{N(0)}}^{\infty}\big\|\chi\Big(\frac{8x}{|t|}\Big)\big\|_{L_x^\frac{2d}{d-3}}
\big\|\chi\Big(\frac{8x}{|\tau|}\Big)\big\|_{L_x^\frac{2d}{d-3}}2^{(d-1)k}\big\langle2^k|\tau-t|\big\rangle^{-d}\,\mathrm{d}t\,\mathrm{d}\tau\\\nonumber
\lesssim&\|u\|_{L^\infty_t\dot
H^{\frac12}_x}^6\int_{-\infty}^{-\frac\delta{N(0)}}\int_{\frac\delta{N(0)}}^{\infty}
(2^k|t|)^{\frac{d-3}2}
(2^k|\tau|)^{\frac{d-3}2}2^{2k}\big\langle2^k|\tau|\big\rangle^{-d/2}\big\langle2^k|t|\big\rangle^{-d/2}\,\mathrm{d}t\,\mathrm{d}\tau\\\nonumber
\lesssim&\int_{-\infty}^{-2^k\frac\delta{N(0)}}\int_{2^k\frac\delta{N(0)}}^{\infty}|t|^{\frac{d-3}2}
|\tau|^{\frac{d-3}2}\big\langle|\tau|\big\rangle^{-d/2}\big\langle|t|\big\rangle^{-d/2}\,\mathrm{d}t\,\mathrm{d}\tau\\\label{equbbest}
\lesssim&\Big(\int_{2^k\frac\delta{N(0)}}^{\infty}\big\langle|t|\big\rangle^{-3/2}dt\Big)^2
 \lesssim\min\Big\{2^{-\frac{k}2}N(0)^\frac12,1\Big\}.
\end{align}

\vskip 0.2in

\noindent Thus, collecting \eqref{equpkvgj}, \eqref{equi1gj}, \eqref{diyabapgj},
\eqref{diyaapgj} and \eqref{equbbest} yields that
\begin{equation*}
a_k(0)^2\lesssim a_k(0)\big(\eta^2\sum_{j\geq
k-3}2^{\frac{k-j}4}a_j+2^{-\frac{k}6}b_k\big)+\big(\eta^2\sum_{j\geq
k-3}2^{\frac{k-j}4}a_j\big)^2+2^{-\frac{k}3}b_k^2+\min\Big\{2^{-\frac{k}2}N(0)^\frac12,1\Big\}.
\end{equation*}
And so we obtain by Young's inequality
$$a_k(0)\lesssim\eta^2\sum_{j\geq
k-3}2^{\frac{k-j}4}a_j+2^{-\frac{k}6}b_k+\min\Big\{2^{-\frac{k}4}N(0)^\frac14,1\Big\}.$$
Combining this inequality with
\eqref{jishuqhgj}, we obtain
\begin{align*}
\alpha_k(0)=&\sum_j2^{-\frac{|j-k|}5}a_j(0)\\
\lesssim&\eta^2\sum_j2^{-\frac{|j-k|}5}\sum_{j_1\geq
j-3}2^{\frac{j-j_1}4}a_{j_1}+\sum_j2^{-\frac{|j-k|}5}2^{-\frac{j}6}b_j
+\sum_j2^{-\frac{|j-k|}5}\min\Big\{2^{-\frac{j}4}N(0)^\frac14,1\Big\}\\
\lesssim&\eta^2\alpha_k+\sum_j2^{-\frac{|j-k|}5}2^{-\frac{j}6}b_j
+\sum_j2^{-\frac{|j-k|}5}2^{-\frac{j}6}\min\Big\{2^{-\frac{1}{12}j}N(0)^\frac14,2^{\frac{j}6}
\Big\}.\end{align*}
 Choosing $\eta$ small and using
\eqref{equ2.93drw}, we deduce that
\begin{align}\label{equalkgj}
\alpha_k(0)\lesssim&\sum_j2^{-\frac{|j-k|}5}2^{-\frac{j}6}b_j+\sum_j2^{-\frac{|j-k|}5}2^{-\frac{j}6}c_j,
\end{align}
where $c_j:=\min\Big\{2^{-\frac{1}{12}j}N(0)^\frac14,2^{\frac{j}6}
\Big\}$ obeys
\begin{equation}\label{equcjgj}
\big\|\{c_j\}_{j\in\Z}\big\|_{l^2}\lesssim N(0)^\frac16.
\end{equation}

\noindent Using Schur's test, \eqref{equalkgj}, \eqref{gjxgjzzx3drwewgj} and
\eqref{equcjgj}, we obtain
\begin{equation}
\Big\|\big\{2^{\frac{k}6}\alpha_k(0)\big\}_{k\in\Z}\Big\|_{l^2}\lesssim
N(0)^\frac16.
\end{equation}
This together with \eqref{ydyxj3drw} ends the proof of  Theorem
\ref{thmpartial}.

\vskip 0.2in

Now we use Theorem \ref{thmpartial} to prove Theorem
\ref{regular}.

\vskip 0.2in

{\bf $\bullet$ Proof of Theorem \ref{regular}.} We first establish a
refined Strichartz estimate in a small time interval.

\begin{lemma}[Refined
Strichartz estimate]\label{lemrfiste} Let
$u:(T_-,+\infty)\times\R^d\to\R$ be a radial solution to problem \eqref{equ1.1} which is
almost periodic modulo symmetries in the sense of Theorem
\ref{three}. Then, there exists a $\delta>0$ sufficiently small such
that for any $t_0\in\R^+$,
\begin{equation}\label{equjbcsqjxiao}
\|u\|_{L_t^qL_x^r(J_\delta(t_0)\times\R^d)}\lesssim
N(t_0)^\frac16,~~\tfrac1q\leq\tfrac{d-1}2\big(\tfrac12-\tfrac1r\big),~\tfrac1q+\tfrac{d}r=\tfrac{d}2-\tfrac23,
\end{equation}
where
$J_\delta(t_0):=\big[t_0-\tfrac{\delta}{N(t_0)},t_0+\tfrac{\delta}{N(t_0)}\big].$
In particularly,
\begin{equation}\label{equjbtb3drw}
\|u\|_{L_t^3L_x^\frac{2d}{d-2}(J_\delta(t_0)\times\R^d)}\lesssim
N(t_0)^\frac16.
\end{equation}
\end{lemma}

\begin{proof} Let
$Y:=L_t^\infty L_x^\frac{6d}{3d-4}\cap
L_t^3L_x^\frac{2d}{d-2}(J_\delta(t_0)\times\R^d).$ We use the
Strichartz estimate to get
\begin{align*}
\|u\|_Y\lesssim&\|\vec u(t_0)\|_{\dot H^\frac23\times\dot
H^{-\frac13}}+\big\|(|x|^{-3}\ast|u|^2)u\big\|_{L_t^1L_x^{\frac{6d}{3d+2}}(J_\delta(t_0)\times\R^d)}\\
\lesssim&
N(t_0)^\frac16+\Big(\frac\delta{N(t_0)}\Big)^{\frac13}\|u\|_{L_t^\frac{9}2L_x^\frac{18d}{9d-16}(J_\delta(t_0)\times\R^d)}^3\\
\lesssim&
N(t_0)^\frac16+\Big(\frac\delta{N(t_0)}\Big)^{\frac13}\|u\|_{Y}^3.
\end{align*}
This implies that
$$N(t_0)^{-\frac16}\|u\|_Y\lesssim1+\delta^\frac13\big(N(t_0)^{-\frac16}\|u\|_Y\big)^3.$$
 Therefore, we obtain by choosing
$\delta>0$ sufficiently small
$$\|u\|_Y\lesssim
N(t_0)^\frac16,$$ as desired.
\end{proof}

As a consequence of \eqref{equjbtb3drw} and the Hardy-Littlewood-Sobolev
inequality, we have the following estimate.

\begin{corollary}\label{cor3drw12}
 Let
$u$ be as in Lemma \ref{lemrfiste}. Then, we have
\begin{equation}\label{equjbtb3drwnl}
\big\|(|x|^{-3}\ast|u|^2)u\big\|_{L_t^1L_x^2(J_\delta(t_0)\times\R^d)}\lesssim
N(t_0)^\frac12.
\end{equation}
\end{corollary}

\vskip 0.2in

Now, we turn to prove Theorem \ref{regular}. By time translation and
recalling $v(t)=u(t)+\frac{i}{\sqrt{-\Delta}}u_t$, we  easily
see that \eqref{addrgrlt} can be reduced to showing
\begin{equation}\label{qmv0gj} \|P_{\leq k}v(0)\|_{\dot H^1}\lesssim
\begin{cases}
 N(0)^{\frac12}\quad \text{if}\quad d\geq5\\
 N(0)^\frac13\quad \text{if}\quad d=4
\end{cases}
\end{equation}
uniformly for all $k\geq k_0>0$.

By Duhamel formula, we have for
 $T_-<T_1<0<T_2<+\infty$
\begin{align}\nonumber
&\|P_{\leq k}v(0)\|_{\dot H^1}^2=\langle P_{\leq k}v(0),P_{\leq
k}v(0)\rangle_{\dot H^1}\\
\nonumber =&\Big\langle P_{\leq
k}\Big(e^{iT_2\sqrt{-\Delta}}v(T_2)\mp\frac{i}{\sqrt{-\Delta}}
\int_0^{T_2}
e^{it\sqrt{-\Delta}}(|x|^{-3}\ast|u|^2)u(t)\,\mathrm{d}t\Big),\\
&\qquad P_{\leq
k}\Big(e^{iT_1\sqrt{-\Delta}}v(T_1)\pm\frac{i}{\sqrt{-\Delta}}
\int_{T_1}^0
e^{i\tau\sqrt{-\Delta}}(|x|^{-3}\ast|u|^2)u(\tau)\,\mathrm{d}\tau\Big)\Big\rangle_{\dot
H^1}\nonumber\\
\nonumber =&\Big\langle P_{\leq k}\Big(\int_0^{T_2}
e^{it\sqrt{-\Delta}}(|x|^{-3}\ast|u|^2)u(t)dt\Big),P_{\leq k}\Big(i
\int_{T_1}^0
e^{i\tau\sqrt{-\Delta}}(|x|^{-3}\ast|u|^2)u(\tau)\,\mathrm{d}\tau\Big)\Big\rangle_{L^2}\\\nonumber
&+\Big\langle e^{iT_2\sqrt{-\Delta}}P_{\leq
k}v(T_2),e^{iT_1\sqrt{-\Delta}}P_{\leq k}v(T_1)\Big\rangle_{\dot
H^1}\\
\nonumber &+\Big\langle P_{\leq k}\Big(\int_0^{T_2}
e^{it\sqrt{-\Delta}}(|x|^{-3}\ast|u|^2)u(t)dt\Big),e^{iT_1\sqrt{-\Delta}}P_{\leq
k}v(T_1)\Big\rangle_{\dot H^1}\\\nonumber &+\Big\langle
e^{iT_2\sqrt{-\Delta}}P_{\leq k}v(T_2),\pm P_{\leq
k}\Big(\frac{i}{\sqrt{-\Delta}} \int_{T_1}^0
e^{i\tau\sqrt{-\Delta}}(|x|^{-3}\ast|u|^2)u(\tau)\,\mathrm{d}\tau\Big)\Big\rangle_{\dot
H^1}\\
\label{equzongd3drw} =:&I_1+I_2+I_3+I_4,
\end{align}
where $\langle f,g\rangle_{\dot H^1}:=\int_{\R^d}
\sqrt{-\Delta}f\cdot\overline{\sqrt{-\Delta}g}\,\mathrm{d}x.$
\vskip0.1cm

First, we claim that
\begin{equation}\label{i234gj}
\lim\limits_{T_2\nearrow \infty}I_2=\lim\limits_{T_2\nearrow
\infty}I_4=0\quad \text{and}\quad \lim\limits_{T_1\searrow
T_-}\lim\limits_{T_2\nearrow \infty}I_3=0.
\end{equation}
Since $\sqrt{-\Delta}P_{\leq k}v(T_1)\in\dot
H^\frac12(\R^d)$ for fixed $T_1$ and $k$, we have by the no-waste
Duhamel formula
\begin{equation}
\lim\limits_{T_2\nearrow \infty}I_2=0.
\end{equation}
On the other hand, it follows from Corollary \ref{cor3drw12} that
$$\int_{-\frac\delta{N(0)}}^{\frac\delta{N(0)}}\big\|P_{\leq k}\big((|x|^{-3}\ast|u|^2)u\big)\big\|_{L_x^2}\,\mathrm{d}t\lesssim
N(0)^\frac12 $$
and
$$\big\|(|x|^{-3}\ast|u|^2)u\big\|_{L_t^1L_x^2([T_1,0]\times\R^d)}<\infty,\;\; ~T_1>T_-.$$
Hence,
$$P_{\leq k}\int_{T_1}^0e^{i\tau\sqrt{-\Delta}}(|x|^{-3}\ast|u|^2)u(\tau)d\tau
\in\dot H^\frac12.$$
Using the no-waste Duhamel formula again, we
obtain
\begin{equation*}
\lim\limits_{T_2\nearrow \infty}I_4=\lim\limits_{T_2\nearrow
\infty}\Big\langle e^{iT_2\sqrt{-\Delta}}P_{\leq k}v(T_2),
  P_{\leq k}\int_{T_1}^0e^{i\tau\sqrt{-\Delta}}(|x|^{-3}\ast|u|^2)u(\tau)\,\mathrm{d}\tau\Big\rangle_{\dot
H^\frac12}=0.
\end{equation*}
Observing that
$$P_{\leq k}\Big(\int_0^{T_2}e^{i\tau\sqrt{-\Delta}}(|x|^{-3}\ast|u|^2)u(\tau)\,\mathrm{d}\tau\Big)=\mp
i\sqrt{-\Delta}P_{\leq k}v(0)\pm
i\sqrt{-\Delta}e^{iT_2\sqrt{-\Delta}}P_{\leq k}v(T_2),$$ one easily
deduces that
\begin{align*}
\lim\limits_{T_1\searrow T_-}\lim\limits_{T_2\nearrow
\infty}I_3=&\lim\limits_{T_1\searrow T_-}\lim\limits_{T_2\nearrow
\infty}\Big\langle e^{iT_1\sqrt{-\Delta}}P_{\leq
k}v(T_1),\sqrt{-\Delta}P_{\leq k}v(0)-
e^{iT_2\sqrt{-\Delta}}P_{\leq k}v(T_2) \Big\rangle_{\dot H^1}\\
=&\lim\limits_{T_1\searrow T_-}\Big\langle
e^{iT_1\sqrt{-\Delta}}P_{\leq k}v(T_1),\sqrt{-\Delta}P_{\leq
k}v(0)\Big\rangle_{\dot
H^1}\\
=&0.
\end{align*}
And so the claim \eqref{i234gj} follows.

It remains to estimate the contribution of the term $I_1$. We claim
\begin{equation}\label{claimi1est}
\Big|\lim\limits_{T_1\searrow T_-}\lim\limits_{T_2\nearrow
T_+}I_1\Big|\lesssim \begin{cases} \big\|P_{\leq k}v(0)\big\|_{\dot
H^1}N(0)^{\frac12}+N(0)\quad \text{if}\quad d\geq5,\\
\big\|P_{\leq k}v(0)\big\|_{\dot
H^1}N(0)^{\frac13}+N(0)^\frac23\quad \text{if}\quad d=4.
\end{cases}.
\end{equation}
Assuming that this claim holds for moment, then we have by
\eqref{i234gj}
\begin{equation}\label{equ2.8I3drw}
\big\|P_{\leq k}v(0)\big\|_{\dot H^1}^2\lesssim
\begin{cases}
\big\|P_{\leq k}v(0)\big\|_{\dot H^1}N(0)^{\frac12}+N(0)\quad
\text{if}\quad d\geq5,\\
\big\|P_{\leq k}v(0)\big\|_{\dot
H^1}N(0)^{\frac13}+N(0)^\frac23\quad \text{if}\quad d=4.
\end{cases}
\end{equation}
Hence by Cauchy's inequality with $\epsilon$, we have
\begin{equation}
\big\|P_{\leq k}v(0)\big\|_{\dot H^1}\lesssim
\begin{cases}
N(0)^{\frac12}\quad \text{if}\quad d\geq5,\\
N(0)^\frac13\quad \text{if}\quad d=4.
\end{cases}
\end{equation}
This implies \eqref{qmv0gj}. Thus, Theorem \ref{regular} follows.

\vskip0.1cm

{\bf $\bullet$ The proof of claim \eqref{claimi1est}.} We first use
Corollary \ref{cor3drw12} to get the small time estimate
\begin{equation}\label{corddd3drw}
\int_{-\frac\delta{N(0)}}^{\frac\delta{N(0)}}\big\|P_{\leq
k}\big((|x|^{-3}\ast|u|^2)u\big)\big\|_{L_x^2}\,\mathrm{d}t\lesssim
N(0)^{\frac12}.
\end{equation}
By the H\"older inequality, the Hardy-Littlewood-Sobolev inequality,
the Sobolev and radial Sobolev embedding, we have for $d\geq5$
$$\big\||x|^\frac{3}2(|x|^{-3}\ast|u|^2)u\big\|_{L_x^2(\R^d)}\lesssim \big\||x|^\frac{3}2u\big\|_{L_x^\frac{2d}{d-4}}\|u\|_{L_x^\frac{2d}{d-1}}^2
\lesssim\|u\|_{\dot H^{\frac12}_x(\R^d)}^3,$$ and for $d=4$
$$\big\||x|^\frac{4}3(|x|^{-3}\ast|u|^2)u\big\|_{L_x^2}\lesssim
\big\||x|^\frac43u\big\|_{L_x^\infty(\R^4)}\|u\|_{L_x^\frac83(\R^4)}^2\leq
\|u\|_{\dot H^\frac23_x(\R^4)}\|u\|_{\dot H^\frac12_x(\R^4)}^2.$$
This implies that
\begin{equation*}
\Big\|(1-\chi)\Big(\frac{8x}{|t|}\Big)(|x|^{-3}\ast|u|^2)u\Big\|_{L_x^2(\R^d)}
\lesssim\begin{cases} \frac1{|t|^\frac32}\|u\|_{\dot
H^\frac12}^3\quad \text{if}\quad d\geq5,\\
\frac1{|t|^\frac43}\|u\|_{\dot H^\frac23}\|u\|_{\dot
H^\frac12}^2\quad \text{if}\quad d=4.
\end{cases}
\end{equation*}
Since $\|u(t,x)\|_{\dot H^\frac23_x(\R^d)}\lesssim
N(t)^\frac16\lesssim1$ for $t\in\R^+$ by Theorem \ref{thmpartial}
and $N(t)\leq1$ when $t\in\R^+$, we have the large time estimate
\begin{align}\label{yuxianxiao3drw}
&\Big\|P_{\leq k}\Big(\int_{\delta/N(0)}^\infty
e^{it\sqrt{-\Delta}}(1-\chi)\Big(\frac{8x}{|t|}\Big)(|x|^{-3}\ast|u|^2)u\,\mathrm{d}t\Big)\Big\|_{L_x^2}\\\nonumber
\lesssim&\int_{\delta/N(0)}^\infty
\Big\|(1-\chi)\Big(\frac{8x}{|t|}\Big)(|x|^{-3}\ast|u|^2)u\Big\|_{L_x^2}\,\mathrm{d}t\\\nonumber
\lesssim&\begin{cases}\delta^{-\frac12}N(0)^\frac12\quad
\text{if}\quad d\geq5,\\
\delta^{-\frac13}N(0)^\frac13\quad \text{if}\quad d=4.
\end{cases}
\end{align}

Define
$$A:=P_{\leq k}\Big(\int_0^{\delta/N(0)}e^{it\sqrt{-\Delta}}(|x|^{-3}\ast|u|^2)u(t)\,\mathrm{d}t+
\int_{\delta/N(0)}^{T_2}
e^{it\sqrt{-\Delta}}(1-\chi)\Big(\frac{8x}{|t|}\Big)(|x|^{-3}\ast|u|^2)u\,\mathrm{d}t\Big)$$
and
$$B:=P_{\leq k}\Big( \int_{\delta/N(0)}^{T_2}
e^{it\sqrt{-\Delta}}\chi\Big(\frac{8x}{|t|}\Big)(|x|^{-3}\ast|u|^2)u\,\mathrm{d}t\Big).$$
Similarly,
$$A':=P_{\leq k}\Big(\int_{-\delta/N(0)}^0e^{i\tau\sqrt{-\Delta}}(|x|^{-3}\ast|u|^2)u(\tau)\,\mathrm{d}\tau+
\int_{T_1}^{-\delta/N(0)}
e^{i\tau\sqrt{-\Delta}}(1-\chi)\Big(\frac{8x}{|\tau|}\Big)(|x|^{-3}\ast|u|^2)u\,\mathrm{d}\tau\Big)$$
and
$$B':=P_{\leq k}\Big( \int_{T_1}^{-\delta/N(0)}
e^{i\tau\sqrt{-\Delta}}\chi\Big(\frac{8x}{|\tau|}\Big)(|x|^{-3}\ast|u|^2)u\,\mathrm{d}\tau\Big).$$
Hence, we have
\begin{equation}\label{equ4.45}
I_1=\langle
A+B,A'\rangle_{L^2}+\langle A,A'+B'\rangle_{L^2}+\langle
B,B'\rangle_{L^2}-\langle A,A'\rangle_{L^2}.
\end{equation}

First, we estimate the contribution of the term  $\langle
A,A'\rangle_{L^2}$. By \eqref{corddd3drw} and
\eqref{yuxianxiao3drw}, we estimate
\begin{equation}\label{equaapgj}
\|A\|_{L_x^2}+\|A'\|_{L_x^2}\lesssim
\begin{cases}N(0)^{\frac12}\quad \text{if}\quad d\geq5\\
N(0)^\frac13\quad \text{if}\quad d=4,
\end{cases}
\Longrightarrow~\langle A,A'\rangle\lesssim \begin{cases} N(0)\quad
\text{if}\quad d\geq5\\
N(0)^\frac23\quad \text{if}\quad d=4.
\end{cases}
\end{equation}

Next, we estimate the contribution of the term  $\langle
B,B'\rangle_{L^2}$. We rewrite $\langle B,B'\rangle_{L^2}$ as
\begin{align*}
&\langle B,B'\rangle_{L^2}\\
=&\Big\langle P_{\leq k}\Big( \int_{\frac\delta{N(0)}}^{T_2}
e^{it\sqrt{-\Delta}}\chi\Big(\frac{8x}{|t|}\Big)(|x|^{-3}\ast|u|^2)u\,\mathrm{d}t\Big),\\
&\qquad\qquad \qquad \qquad\qquad \qquad P_{\leq
k}\Big( \int_{T_1}^{-\frac\delta{N(0)}}
e^{i\tau\sqrt{-\Delta}}\chi\Big(\frac{8x}{|\tau|}\Big)(|x|^{-3}\ast|u|^2)u\,\mathrm{d}\tau\Big)\Big\rangle_{L^2}\\
=&\int_{T_1}^{-\frac\delta{N(0)}}\int_{\frac\delta{N(0)}}^{T_2}\Big\langle
\chi\Big(\frac{8x}{|t|}\Big)(|x|^{-3}\ast|u|^2)u,P_{\leq k}^2\Big(
e^{i(\tau-t)\sqrt{-\Delta}}\chi\Big(\frac{8x}{|\tau|}\Big)(|x|^{-3}\ast|u|^2)u\Big)\Big\rangle_{L_x^2}\,\mathrm{d}t\,\mathrm{d}\tau.
\end{align*}
In the same way as deriving in  \eqref{est:kkxy}, we  estimate
the kernel of $P_{\leq k}^2 e^{i(\tau-t)\sqrt{-\Delta}}$ as follows
$$\big|\tilde K_k(x)\big|\lesssim 2^{dk}\big\langle2^k|\tau|\big\rangle^{-d/2}\big\langle2^k|t|\big\rangle^{-d/2}.$$
Combining this inequality with the proof process of \eqref{equbbest}, we
obtain for $2^k\gg N(0)$
\begin{equation}\label{equbbpgj}
\begin{split}
 \langle B,B'\rangle_{L^2}\lesssim&2^k\Big(\int_{2^k\frac\delta{N(0)}}^{\infty}\big\langle|t|\big\rangle^{-3/2}dt\Big)^2\lesssim N(0).
\end{split}
\end{equation}

Finally, we only need to estimate the contribution of $\langle
A,A'+B'\rangle_{L^2}$, since $\langle A+B,A'\rangle_{L^2}$ has the same estimate by symmetry. Observing that
$$P_{\leq k}\Big(\int_{T_1}^0e^{i\tau\sqrt{-\Delta}}(|x|^{-3}\ast|u|^2)u(\tau)\,\mathrm{d}\tau\Big)=\mp
i\sqrt{-\Delta}P_{\leq k}v(0)\pm
i\sqrt{-\Delta}e^{iT_1\sqrt{-\Delta}}P_{\leq k}v(T_1),$$
we obtain
\begin{equation}\label{eequaabppgj}
\langle A,A'+B'\rangle_{L^2}=\Big\langle A,\mp
i\sqrt{-\Delta}P_{\leq k}v(0)\pm
i\sqrt{-\Delta}e^{iT_1\sqrt{-\Delta}}P_{\leq
k}v(T_1)\Big\rangle_{L^2}.
\end{equation}
Using \eqref{equaapgj}, we have
\begin{equation}\label{equaabpartgj}
\Big|\Big\langle A, \sqrt{-\Delta}P_{\leq
k}v(0)\Big\rangle_{L_x^2}\Big|\leq\|A\|_{L_x^2}\big\|\sqrt{-\Delta}P_{\leq
k}v(0)\big\|_{L_x^2}\lesssim \begin{cases}\|P_{\leq k}v(0)\|_{\dot
H^1}N(0)^{\frac12}\quad \text{if}\quad d\geq5\\
\|P_{\leq k}v(0)\|_{\dot H^1}N(0)^{\frac13}\quad \text{if}\quad d=4.
\end{cases}
\end{equation}
We claim that
\begin{equation}\label{claimaaparttoz}
\lim_{T_1\searrow T_-}\lim_{T_2\nearrow \infty}\Big|\Big\langle A,
\sqrt{-\Delta}e^{iT_1\sqrt{-\Delta}}P_{\leq
k}v(T_1)\Big\rangle_{L_x^2}\Big|=0.
\end{equation}
In fact, since  $\|(-\Delta)^\frac14A\|_{L_x^2}\lesssim 2^\frac{k}2
N(0)^{\frac13}$ by \eqref{equaapgj}, we deduce that for a fixed $k$
and letting $T_2\nearrow \infty$, $(-\Delta)^\frac14 A$ converges to
$$ (-\Delta)^\frac14P_{\leq
k}\Big(\int_0^{\frac\delta{N(0)}}e^{it\sqrt{-\Delta}}(|x|^{-3}\ast|u|^2)u(t)\,\mathrm{d}t+
\int_{\frac\delta{N(0)}}^{\infty}
e^{it\sqrt{-\Delta}}(1-\chi)\Big(\frac{8x}{|t|}\Big)(|x|^{-3}\ast|u|^2)u\,\mathrm{d}t\Big),$$ in $L_x^2(\R^d).$
Then, we obtain by the no-waste Duhamel formula
\begin{equation}\label{equ4.51}
\lim\limits_{T_1\searrow T_-}\lim\limits_{T_2\nearrow
T_+}\Big\langle A, \sqrt{-\Delta}e^{iT_1\sqrt{-\Delta}}P_{\leq
k}v(T_1)\Big\rangle_{L_x^2}=0,
\end{equation}
as desired.

Collecting \eqref{equ4.45}-\eqref{equbbpgj} and  \eqref{equ4.51}, we obtain \eqref{equ2.8I3drw}.
Therefore, we conclude Theorem \ref{regular}.

\subsection{Proof of  Proposition
\ref{propgdzzx}} \; First, Theorem \ref{regular} implies that
$$\big\|(u,u_t)(t)\big\|_{L_t^\infty(\R,\dot H^1\times L^2)}\lesssim
1.$$ Hence, by the same argument in spirit of Lemma 6.12 in
\cite{KVnote}, Proposition \ref{propgdzzx} can be reduced to show
that
\begin{equation}\label{gaojzzx}
\| (u,u_t)(t)\|_{L_t^\infty(\R,\dot H^2\times\dot H^1)}\lesssim 1.
\end{equation}
Let $v(t)$ be defined as in \eqref{defvt}. Then, \eqref{gaojzzx} can
be reduced to show that
\begin{equation}\label{qmvh2es}
\big\|P_{\leq k}v(0)\big\|_{\dot H^2}\lesssim1
\end{equation}
uniformly for all $k>k_0$. The proof is the same as in the proof of
Theorem \ref{regular}.  Employing the No-waste Duhamel formula as in
the proof of Theorem \ref{regular}, we only need to prove the
following double Duhamel term obeying
\begin{equation*}
\Big|\Big\langle P_{\leq
k}\Big(\int_{T_1}^0e^{it\sqrt{-\Delta}}\nabla\big((|x|^{-3}\ast|u|^2)u\big)(t)\,\mathrm{d}t\Big),
 P_{\leq
k}\Big(\int_0^{T_2}e^{i\tau\sqrt{-\Delta}}\nabla\big((|x|^{-3}\ast|u|^2)u\big)(\tau)\,\mathrm{d}\tau\Big)
 \Big\rangle_{L_x^2}\Big|\lesssim1
\end{equation*}
uniformly for $T_1<0<T_2$. To do this, we suffice to show the similar estimate as
in \eqref{corddd3drw}, \eqref{yuxianxiao3drw} and \eqref{equbbpgj}.

 For this purpose,  we first establish the following refined Strichartz
estimate in a small time interval.

\begin{lemma}\label{adrlemh2}  Let
$u:\R\times\R^d\to\R$ be a radial solution to problem \eqref{equ1.1} which is almost periodic
modulo symmetries in the sense of Theorem \ref{three} with
$N(t)\equiv1$. Then, there exists a $\delta>0$ such that for any
$t_0\in\R$ and for $J:=(t_0-\delta,t_0+\delta)$, we have
\begin{equation}\label{equ3.13drw}
\|u\|_{L_t^2L_x^\frac{2d}{d-3}(J\times\R^d)}\lesssim1.
\end{equation}
\end{lemma}
\begin{proof} Without loss of generality, we assume that $t_0=0$.
 Let
 $$Z(J)=L_t^\infty(J,\dot
H^1)\cap L_t^2(J,L_x^\frac{2d}{d-3}).$$
 We have by Strichartz
estimate and  Theorem \ref{regular}
\begin{align*}
\|u\|_{Z(J)}\lesssim&\| (u_0,u_1)\|_{\dot H^1\times
L^2}+\big\|(|x|^{-3}\ast|u|^2)u\big\|_{L_t^1L_x^2(J\times\R^d)}\\
\lesssim&\|(u_0,u_1)\|_{\dot H^1\times
L^2}+\|u\|_{L_t^3L_x^\frac{2d}{d-2}(J\times\R^d)}^3\\
\lesssim&\|(u_0,u_1)\|_{\dot H^1\times
L^2}+\delta\|u\|_{L_t^\infty(J,\dot H^1)}^3\\
\lesssim& 1+\delta \|u\|_{Z(J)}^3.
\end{align*}
  This ends the proof by choosing $\delta>0$ sufficiently small.
\end{proof}

\vskip 0.2in

Using Lemma \ref{adrlemh2}, we have the small time estimate
\begin{equation}\label{smalltest}
\begin{split}
&\Big\|P_{\leq
k}\Big(\int_0^{\delta}e^{i\tau\sqrt{-\Delta}}\nabla\big((|x|^{-3}\ast|u|^2)u\big)
(\tau)\,\mathrm{d}\tau\Big)\Big\|_{L_x^2}\\
\lesssim&\int_0^\delta\big\|\nabla\big((|x|^{-3}\ast|u|^2)u\big)\big\|_{L_x^2}\,\mathrm{d}\tau\lesssim
\|\nabla u\|_{L_t^\infty
L_x^2}\|u\|_{L_t^2L_x^\frac{2d}{d-3}([0,\delta))}^2\lesssim1.
\end{split}
\end{equation}
On the other hand, when $d\geq5$, by the radial Sobolev embedding:
$\big\||x|^\frac32u\big\|_{L_x^\frac{2d}{d-4}(\R^d)}\lesssim\|u\|_{\dot
H^\frac12(\R^d)},$ the H\"older inequality and the
Hardy-Littlewood-Sobolve inequality, we obtain
\begin{align*}
\Big\|(1-\chi)\Big(\frac{8x}{|t|}\Big)\big(|x|^{-3}\ast(\nabla|u|^2)\big)u
\Big\|_{L_x^2}\lesssim&\Big\|(1-\chi)\Big(\frac{8x}{|t|}\Big)u\Big\|_{L_x^\frac{2d}{d-4}}\big\||x|^{-3}\ast(u\nabla
u)\big\|_{L_x^\frac{d}2}\\
\lesssim&|t|^{-\frac32}\big\||x|^\frac32u\big\|_{L_x^\frac{2d}{d-4}}\|\nabla
u\|_{L_x^2}\|u\|_{L_x^\frac{2d}{d-2}}\\
\lesssim&|t|^{-\frac32}.
\end{align*}
When $d=4$, by the radial Sobolev embedding:
$\big\||x|^\frac43u\big\|_{L_x^\infty(\R^4)}\lesssim\|u\|_{\dot
H^\frac23(\R^4)}$, we obtain
$$\Big\|(1-\chi)\Big(\frac{8x}{|t|}\Big)\big(|x|^{-3}\ast(\nabla|u|^2)\big)u
\Big\|_{L_x^2}\lesssim\Big\|(1-\chi)\Big(\frac{8x}{|t|}\Big)u\Big\|_{L_x^\infty}\big\||x|^{-3}\ast(u\nabla
u)\big\|_{L_x^2}\lesssim|t|^{-\frac43}.$$
 Similarly, when $d\geq5,$
using the radial Sobolev embedding:
$\big\||x|^{\frac32}f\big\|_{L_x^\infty(\R^d)}\lesssim\big\||\nabla|^\frac12
f\big\|_{L_x^\frac{d}2(\R^d)},$ the H\"older inequality and the
Hardy-Littlewood-Sobolve inequality, we get
\begin{align*}
\Big\|(1-\chi)\Big(\frac{8x}{|t|}\Big)\big(|x|^{-3}\ast(|u|^2)\big)\nabla
u \Big\|_{L_x^2(\R^d)}\lesssim&\|\nabla
u\|_{L_x^2}\Big\|(1-\chi)\Big(\frac{8x}{|t|}\Big)\big(|x|^{-3}\ast(|u|^2)\big)\Big\|_{L_x^\infty(\R^d)}\\
\lesssim&|t|^{-\frac32}\Big\||x|^{\frac32}\big(|x|^{-3}\ast(|u|^2)\big)\Big\|_{L_x^\infty(\R^d)}\\
\lesssim&|t|^{-\frac32}\Big\||\nabla|^\frac12\big(|x|^{-3}\ast|u|^2\big)\Big\|_{L_x^\frac{d}2(\R^d)}\\
\lesssim&|t|^{-\frac32}\|\nabla
u\|_{L_x^2(\R^d)}\big\||\nabla|^\frac12u\big\|_{L_x^2(\R^d)},
\end{align*}
and when $d=4$, by
$\big\||x|^\frac43u\big\|_{L_x^\infty(\R^4)}\lesssim\|u\|_{\dot
H^\frac23(\R^4)}$, we have
$$\Big\|(1-\chi)\Big(\frac{8x}{|t|}\Big)\big(|x|^{-3}\ast(|u|^2)\big)\nabla
u \Big\|_{L_x^2(\R^4)}\lesssim |t|^{-\frac43}.$$

\noindent Hence, we have the large time estimate
\begin{equation}\label{largtest}
\begin{split}
&\Big\|P_{\leq
k}\Big(\int_\delta^{T_2}e^{i\tau\sqrt{-\Delta}}(1-\chi)\Big(\frac{8x}{|t|}\Big)
\nabla\big((|x|^{-3}\ast|u|^2)u\big)(\tau)\,\mathrm{d}\tau\Big)\Big\|_{L_x^2(\R^d)}\\
\lesssim&\int_\delta^\infty
\Big\|(1-\chi)\Big(\frac{8x}{|t|}\Big)\nabla\big((|x|^{-3}\ast|u|^2)u\big)\Big\|_{L_x^2(\R^d)}\,\mathrm{d}\tau\\\lesssim&
\begin{cases}
\delta^{-\frac12}\quad \text{if}\quad d\geq5\\
\delta^{-\frac13}\quad \text{if}\quad d=4.
\end{cases}
\end{split}
\end{equation}

In  the same way  as in the proof of \eqref{equbbpgj}, we
deduce that for $2^k\gg1$
\begin{align}\nonumber
&\Big|\int_\delta^{T_2}\int_{T_1}^{-\delta}\Big\langle P_{\leq
k}\Big(e^{it\sqrt{-\Delta}}\nabla\big((|x|^{-3}\ast|u|^2)u\big)
(t)\Big),P_{\leq
k}\Big(e^{i\tau\sqrt{-\Delta}}\nabla\big((|x|^{-3}\ast|u|^2)u\big)(\tau)\Big)
 \Big\rangle_{L_x^2}\,\mathrm{d}t\,\mathrm{d}\tau\Big|\\\label{zeroh}
 \lesssim&1.
\end{align}
Using the same argument as deriving
\eqref{qmv0gj}, we obtain \eqref{qmvh2es} by making use of \eqref{smalltest}, \eqref{largtest} and
\eqref{zeroh} to replace  \eqref{corddd3drw},
\eqref{yuxianxiao3drw} and \eqref{equbbpgj}. Therefore, we conclude
Proposition \ref{propgdzzx}.

\vskip 0.5in

{\bf Acknowledgements:} The authors would like to express their
gratitude to  the anonymous referee for their invaluable comments
and suggestions.  C. Miao was supported by the NSFC under grant
No.11171033 and 11231006, and by Beijing Center of Mathematics and
Information  Interdisciplinary Science. J. Zhang was supported by
the Beijing Natural Science Foundation(1144014) and National Natural
Science Foundation of China (11401024), Excellent young scholars
Research Fund of Beijing Institute of Technology. J. Zheng was
partly supported by the European Research Council, ERC-2012-ADG,
project number 320845: Semi-Classical Analysis of Partial
Differential Equations.

\begin{center}

\end{center}

\end{document}